\documentclass[12pt]{amsart}
\usepackage[latin1]{inputenc}
\usepackage{amsmath,amsthm,amssymb}
\usepackage{mathrsfs}
\usepackage{color}
\usepackage[colorlinks]{hyperref}

\usepackage{enumerate}

\usepackage{algorithmic}
\usepackage[ruled]{algorithm}
\usepackage{caption}
\usepackage[sort&compress,comma,authoryear]{natbib}

\newcommand{\Proj}{\textnormal{Proj}\,}
\newcommand{\Spec}{\textnormal{Spec}\,}

\newcommand{\sat}{{\textnormal{sat}}}

\newcommand{\plane}{\mathcal{H}}

\newcommand{\St}{\mathrm{St}}

\newcommand{\supp}{\mathrm{supp}}

\newcommand{\Ht}{\mathrm{Ht}}
\newcommand{\cN}{\mathcal{N}\,}
\newcommand{\Nf}{\mathrm{Nf}}

\newcommand{\In}{\mathrm{in}}

\newcommand{\F}{\underline{\mathrm{L}}_{Y,J_{\geq m}}^{p(t)}}

\numberwithin{equation}{section}

\newtheorem{lemma}{Lemma}[section]
\newtheorem{theorem}[lemma]{Theorem}
\newtheorem{corollary}[lemma]{Corollary}
\newtheorem{proposition}[lemma]{Proposition}

\theoremstyle{definition}

\newtheorem{remark}[lemma]{Remark}

\newtheorem{example}[lemma]{Example}
\newtheorem{definition}[lemma]{Definition}
\DeclareMathAlphabet{\mathpzc}{OT1}{pzc}{m}{it}

\newcommand{\Aox}{A[\mathbf x,x_n]}

\begin{document}

\title{Functors of liftings of projective schemes}

\author[C. Bertone]{Cristina Bertone}
\email{cristina.bertone@unito.it}
\address{Dip. di Matematica, Universit\`a di Torino, Torino, Italy}

\author[F. Cioffi]{Francesca Cioffi}
\email{cioffifr@unina.it}
\address{Dip. di Matematica e Applicazioni, Universit\`a degli Studi di Napoli Federico II, Napoli, Italy}

\author[D. Franco]{Davide Franco}
\email{davide.franco@unina.it}
\address{Dip. di Matematica e Applicazioni, Universit\`a degli Studi di Napoli Federico II, Napoli, Italy}

\begin{abstract}
A classical approach to investigate a closed projective scheme $W$ consists of  
considering a ge\-ne\-ral hyperplane section of $W$, which  inherits many properties of $W$. The inverse problem that consists in finding a scheme $W$ starting from a possible hyperplane section $Y$ is called a {\em lifting problem}, and every such scheme $W$ is called a {\em lifting} of $Y$.  
Investigations in this topic can produce methods to obtain  schemes with specific properties. For example, any smooth point for $Y$ is smooth also for $W$.

We characterize all the liftings of $Y$ with a given Hilbert polynomial by a parameter scheme that is obtained by gluing suitable affine open subschemes in a Hilbert scheme and is described  through the functor it represents. 
We use constructive methods  from Gr\"obner and marked bases theories. Furthermore, by classical tools we obtain an analogous result for equidimensional liftings. Examples of explicit computations are provided.
\end{abstract}



\maketitle


\section*{Introduction}

Let $K$ be an infinite field, $A$ be a Noetherian $K$-algebra and $\mathbb P^n_A$ the $n$-dimensional projective space over $A$. A classical approach to investigate a closed projective scheme $W$ consists of 
considering a general hyperplane section of $W$, because many properties of $W$  
are preserved under general hyperplane sections and can be easier recognized in subschemes of lower dimension. 

The inverse problem that consists in finding a scheme $W$ starting from a possible hyperplane section is called a \emph{lifting problem} and investigations in this topic can produce methods to obtain affine or projective schemes with specific properties. 

To address the question, we recall that in \citep{Macaulay} Macaulay lifted monomial ideals to obtain sets of points with a given Hilbert functions, in \citep{H66} Hartshorne proved the connectedness of a Hilbert scheme by lifting Borel ideals and in \citep{RA} Reeves computed the radius of a Hilbert scheme with an analogous procedure (distractions). Moreover, by $t$-liftings and pseudo-liftings in \citep{MiNa} Migliore and Nagel obtained special configurations of linear varieties (stick figures). Much interest has been in particular given to the study of the $x_n$-liftings, which can be defined in terms of ideals (see \citep{GGR, Roitman}) or equivalently in terms of $K$-algebras, like proposed by Grothendieck (see \citep{BuchEis,Roitman} and the references therein). 
Starting from the papers \citep{Laudal} and \citep{Strano}, many authors also give significant contributions to find conditions on some invariants of a variety $W$ so that a degree $d$ hypersurface containing a hyperplane section of $W$ lifts to a hypersurface containing $W$.

In this paper, we consider the following lifting problem: given a closed subscheme $Y\subset \mathbb P^{n-1}_K$, explicitly describe all closed subschemes $W\subset \mathbb P^n_A$ such that $Y$ is a general hyperplane section of $W$, up to an extension of scalars. 
Every such scheme $W$ is called a {\em lifting} of $Y$ over $A$ (Definition \ref{def:geometric lifting}) and the saturated defining ideal $I$ of $W$ is called a {\em lifting} of the saturated defining ideal $I'$ of $Y$ (see Definition \ref{def:geometric lifting per gli ideali} and Proposition \ref{prop:h generico}). By the definition, we almost suddenly obtain that $W$ is smooth at every point at which $Y$ is smooth (see Proposition \ref{prop:smooth points}).

In order to have a better understanding of the issue we have in mind, let us consider a double point $P\in \mathbb P^n_K$. It is very well known that, for any negative integer $g$, there exist double lines $C_g\subset \mathbb P^3_K$ with arithmetic genus $g$ and having $P$ as general hyperplane section. So, even when $Y\subset \mathbb P^{n-1}_K$ is let to vary in a concrete quasi-projective scheme, we cannot hope that the liftings of $Y$ are parameterized 
by a quasi-projective scheme. In view of this we are led to set the Hilbert polynomial $p_Y(t)$ of $Y$ and look for a functor describing all the liftings of $Y$ with a given Hilbert polynomial $p(t)$ such that its first difference $\Delta p(t):=p(t)-p(t-1)$ 
is equal to $p_Y(t)$.

The framework of the present paper is both functorial and constructive. 
We use computational methods that are borrowed from Gr\"obner and marked bases theories and which involve quasi-stable ideals (see \cite{LR,LR2,BCR2}). 
Throughout the paper, a functor $F:{\text{\rm Noeth-}K\text{\rm -Alg}}\rightarrow {\mathrm{Sets}}$ will be said \emph{representable} if there is a scheme $X$ and a natural isomorphism between $F$ and the functor of points $h_X=\mathrm{Hom}( - ,X)$ applied on affine schemes.
If $X$ is an affine scheme, this definition coincides with the usual one in category theory (see, for instance, \cite[Definition 2.1]{Vistoli}).

Denoting by $\mathrm{Hilb}_{p(t)}^n$ the Hilbert scheme parameterizing all the subschemes of $\mathbb P^n_K$ with Hilbert polynomial $p(t)$, our main results, which are collected in Theorems \ref{th:schema che rappresenta} and \ref{th:funtore di punti}, can be summarized in the following way.
\vskip 1mm
\noindent{\bf Theorem A.} {\em  Let $p(t)$ be the Hilbert polynomial of a lifting of $Y$.
\begin{itemize}
\item[(i)] If $Y$ is equidimensional, then the family of the equidimensional liftings of $Y$ with Hilbert polynomial $p(t)$ is parameterized by a locally closed subscheme of $\mathrm{Hilb}_{p(t)}^n$.
\item[(ii)] The family of the liftings of $Y$ with Hilbert polynomial $p(t)$ is parameterized by a subscheme of $\mathrm{Hilb}_{p(t)}^n$ which can be explicitly constructed.
\end{itemize}
}
This paper has been motivated by the investigation of $x_n$-liftings of a homogeneous polynomial ideal (see Definition \ref{def:lifting}) that has been faced in \citep{BCGR} from a functorial point of view. However, the question treated here is different and presents new problems to be solved.
We now give a detailed outline of the contents of the paper.

In what follows, we consider $K[x_0,\dots,x_{n-1}]$ as a subring of $A[x_0,\dots,x_{n-1},x_n]$ and the variable $x_n$ is \emph{generic} for every ideal in $A[x_0,\dots,x_n]$ we take. 
This assumption allows us to exploit the behavior of Gr\"obner bases with respect to the degree reverse lexicographic order when we need the saturation of homogeneous ideals and of their initial ideals (see Remark \ref{rem:generic and degrevlex}). The study of $x_n$-lifting presented in \citep{BCGR} needed weaker hypotheses on the term order that was indeed chosen in a more general class of term orders (see Definition 4.4.1 in \citep{KR2}, and \citep{erdos} for a first generic classification of term orderings). 

Here, we only consider the degrevlex term order. In this setting, we recall the notion of Gr\"obner functor and Gr\"obner stratum (see Definition \ref{def:GrStratum}) 
together with the functor of $x_n$-liftings. Then, we prove that our Definition \ref{def:geometric lifting} of lifting of a projective scheme is equivalent to Definition \ref{def:geometric lifting per gli ideali} of lifting of a saturated polynomial ideal (see Proposition \ref{prop:h generico}). This allows to observe that 
the algebraic counterpart of our lifting problem consists in identifying the liftings $W$ of $Y$ with the $x_n$-liftings of ideals having saturation equal to $I'$ (see Proposition \ref{prop:geometric lifting}). This characterization implies that all the liftings with a given Hilbert polynomial $p(t)$ can be identified with the points of a disjoint union of locally closed subschemes in the Hilbert scheme $\mathrm{Hilb}_{p(t)}^n$ via suitable Gr\"obner strata (see  
Theorem \ref{th:parametrizzazione}).

Moreover, in the further non-restrictive hypothesis that the variable $x_{n-1}$ is generic for $I'$, we prove that a subscheme $W$ is a lifting of $Y$ only if $I$ belongs to a Gr\"obner stratum over a monomial ideal $J$ which is a lifting of the initial ideal of $I'$ (see Theorem \ref{th:dove}). The proofs of these results are constructive and then produce a method for the computation of the locally closed subschemes in a Hilbert scheme whose disjoint union corresponds to all liftings with a given Hilbert polynomial via Gr\"obner strata (see Algorithm \ref{alg:computation in GS}). 

So, we obtain embeddings of liftings with a given Hilbert polynomial in Gr\"obner strata and, hence, in a Hilbert scheme. Then, it is natural to look at liftings from a functorial point of view. Thanks to the constructive characterization of liftings given in Section \ref{sec:Construction}, we are finally able to define functors related to our liftings as subfunctors of a Hilbert functor (Definitions \ref{def:funtore lifting} and \ref{def:funtore lifting eq}). Given a Hilbert polynomial $p(t)$, we prove that the functor $\underline{\mathrm{L}}_Y^{p(t)}$ of liftings of $Y$ and the functor $\underline{\mathrm{L}}_Y^{p(t),e}$ of equidimensional liftings with Hilbert polynomial $p(t)$ are representable. 

For what concerns the functor $\underline{\mathrm{L}}_Y^{p(t),e}$, we adapt to our situation classical arguments of algebraic geometry, like the upper semicontinuity of the dimension of the fibers of a dominant map (Theorem \ref{th:schema che rappresenta}). 

For what concerns the functor $\underline{\mathrm{L}}_Y^{p(t)}$, we observe that the locally closed subfunctors that are represented by the locally closed schemes above introduced in suitable Gr\"obner strata are not necessarily open subfunctors of $\underline{\mathrm{L}}_Y^{p(t)}$ and so are not  expected to be suitable to give a unique scheme representing $\underline{\mathrm{L}}_Y^{p(t)}$. At this point our constructive approach is pushed forward by the following fact. Given the degrevlex term order, if the initial ideal of the saturated ideal $I'$ defining $Y$ is quasi-stable then the initial ideal of the saturated ideal $I$ defining a lifting of $Y$ is quasi-stable (see Theorem \ref{th:lifting}). This result, which is achieved in Section \ref{sec:quasi-stable}, allows us to study liftings in marked schemes over truncations of quasi-stable ideals, which are open subschemes in a Hilbert scheme (see Definition \ref{def:markedscheme}). Then, we are able to replace the disjoint locally closed subschemes described by means of Gr\"obner strata by open subschemes that we describe by means of marked schemes.
Indeed, we exploit the features of marked schemes in order to construct an open covering of the functor  $\underline{\mathrm{L}}_Y^{p(t)}$ that provides a scheme representing  $\underline{\mathrm{L}}_Y^{p(t)}$ by a gluing procedure (see Lemma \ref{lemma:sezioni ideali marcati}, Theorems \ref{th:aperti} and \ref{th:funtore di punti} and Remark \ref{rem:stromme}). 
The proof of this last result is constructive and gives rise to Algorithm~\ref{alg:computation in MS}. In Example \ref{ex:esempio bello} we exhibit an application of this construction.


\section{Setting}

We consider commutative unitary rings and morphisms that preserve the unit. 
Given sets of variables $\mathbf x=\lbrace x_0,\dots,x_{n-1}\rbrace$ and $\mathbf x,x_n=\lbrace x_0,\dots,x_{n}\rbrace$, we assume they are ordered as $x_0>\cdots>x_n$. 
For a term $x^\alpha:=x_0^{\alpha_0}\cdots x_n^{\alpha_n}$ 
other than $1$, we denote by $\min(x^\alpha)$ the smallest variable appearing in $x^\alpha$ with a non-zero exponent 
and by $\deg(x^\alpha)=\vert\alpha\vert:=\sum_i \alpha_i$ the degree of $x^\alpha$.

Let $K$ be an infinite field. We denote the polynomial ring $K[x_0,\dots,$ $x_{n-1}]$ by $K[\mathbf x]$ and the polynomial ring $K[x_0,\dots,x_n]$ 
by $K[\mathbf x,x_n]$. For any (Noetherian) $K$-algebra $A$, $A[\mathbf x]$ denotes the polynomial ring $K[\mathbf x]\otimes_K A$ and $\Aox$ denotes 
$K[\mathbf x,x_n]\otimes_K A$ as \emph{standard graded algebras}. 
Obviously, $A[\mathbf x]$ is a subring of $\Aox$, hence the following notations and assumptions will be stated for $\Aox$ but will hold for $A[\mathbf x]$ too. 

For any non-zero homogeneous polynomial $f \in A[\mathbf x,x_n]$, the {\em support} $\supp(f)$ of $f$ is the set of terms in the variables $\mathbf x,x_n$ that appear in $f$ with a non-zero coefficient.
We denote by $\mathrm{coeff}(f)\subset A$ the set of the coefficients in $f$ of the terms of $\supp(f)$. For any subset $\Gamma\subseteq A[\mathbf x,x_n]$, $\Gamma_t$ is the set of homogeneous polynomials of $\Gamma$ of degree $t$. Furthermore, we denote by $\langle \Gamma\rangle$ the $A$-module generated by $\Gamma$. 
When $\Gamma$ is a homogeneous ideal, we denote by $\Gamma_{\geq t}$ the ideal  containing 
the homogeneous polynomials of $\Gamma$ of degree $\geq t$. 

Let $I$ be a homogeneous ideal of $A[\mathbf x,x_n]$. 
The {\em saturation} of $I$ is $I^\sat=\{f\in A[\mathbf x,x_n] \ \vert \ \forall \ i=0,\dots n, \exists \ k_i : x_i^{k_i} f \in I\}$. The ideal $I$ is {\em saturated} if $I=I^{\sat}$ and is {\em $m$-saturated} if $I_t=(I^\sat)_t$ for every $t\geq m$. The \emph{satiety} of $I$, denoted by $\sat(I)$, is the smallest $m$ for which $I$ is $m$-saturated.
A linear form $h\in A[\mathbf x,x_n]$ is said \emph{generic} for $I$ if $h$ is not a zero-divisor in $A[\mathbf x,x_n]/I^\sat$.

A {\em monomial ideal} $J$ of $A[\mathbf x,x_n]$ is an ideal generated by terms. We denote by $B_J$ the minimal set of terms generating $J$ and by $\mathcal N(J)$ the sous-escalier of $J$, i.e.~the set of terms outside $J$. 

\begin{definition}\label{def:stable} 
A monomial ideal $J\subset A[\mathbf x,x_n]$ is \emph{quasi-stable} if for every $x^\alpha \in J$ and $x_j>\min(x^\alpha)$, there is $t\geq 0$ such that $\frac{x_j^t x^\alpha}{\min(x^\alpha)}$ belongs to $J$.
\end{definition}

Given a monomial ideal $J$ and an ideal $I$, a {\em $J$-reduced form modulo $I$} of a polynomial $f$ is a polynomial $\bar f$ such that $f-\bar f$ belongs to $I$ and $\supp(\bar f)$ is contained in $\mathcal N(J)$.
If $\bar f$ is the unique possible $J$-reduced form modulo $I$ of $f$, then it is called the {\em $J$-normal form modulo $I$} of $f$ and is denoted by $\Nf(f)$.

With the usual language of Gr\"obner bases theory, {\em from now,} we consider the degree reverse lexicographic term order $\succ$ (degrevlex, for short) and, for every non-zero polynomial $f\in A[\mathbf x,x_n]$, denote by $\In(f)=\max_{\succ}\supp(f)$ its initial term and by $c_f$ the coefficient of $\In(f)$ in $f$. For every ideal $I\subset A[\mathbf x,x_n]$ let $\In(I)=(c_f \In(f) : f\in I)$ be its initial ideal. A set $\{f_1,\dots,f_t\}\subset I$ is a Gr\"obner basis of $I$ if $\{c_{f_1}\In(f_1),\dots,c_{f_t}\In(f_t)\}$ generates $I$. 
In general, Gr\"obner bases theory over rings \citep{CM,Z,M} is more intricate than over fields (see for instance the detailed discussion in \citep{Lederer2011}), and the possibly non-invertible leading coefficients make Gr\"obner bases over rings not well-suited for functorial constructions.  This is the reason why
in this paper we consider homogeneous ideals $I$ generated by either Gr\"obner bases in $K[\mathbf x,x_n]$ or \emph{monic} Gr\"obner bases in $A[\mathbf x,x_n]$ or \emph{marked bases} over quasi-stable ideals in $K[\mathbf x,x_n]$ and $A[\mathbf x,x_n]$ (see Section \ref{sec:backgroundMarkedBases} for the definition of marked basis).
So, the quotients $\Aox/I$ are free graded $A$-modules and 
this is a very key point for the use of functors we will introduce (see \citep{BGS91}, \citep[Lemma 6.1]{BCR2}, \citep[Section 5]{LR2}). 
The crucial fact is that, if $\varphi \colon A \rightarrow B$ is a morphism of $K$-algebras and $I$ is an ideal in $A[\mathbf x]$ (or $\Aox$) generated by a {\em monic} Gr\"obner basis (resp.~a marked basis over a quasi-stable ideal) $G_I$, then $I\otimes_A B$ is generated by $\varphi(G_I)$ which is again a {\em monic} Gr\"obner basis (resp.~a marked basis over a quasi-stable ideal) 
(see Proposition~\ref{prop:funtore}). 

In this setting, for a homogeneous ideal $I$ of $\Aox$ we can consider the {\it Hilbert function} $h_{\Aox/I}$ of the free graded $A$-module $\Aox/I$ and its {\it Hilbert polynomial} $p_{\Aox/I}(t)$ like the Hilbert function and the Hilbert polynomial of $K[\mathbf x, x_n]/\In(I)$. When we say \lq\lq ideal $I$ with Hilbert polynomial $p(t)$\rq\rq \ we mean $p_{\Aox/I}(t)=p(t)$. For the geometric definition of Hilbert polynomial of a projective scheme over a field we refer to \citep[Chapter III, Exercise 5.2]{H}. 

\begin{remark}\label{rem:generic and degrevlex} 
\
\begin{enumerate}
\item If $h$ is a generic linear form for $I$, then $I^\sat=(I : h^\infty):=\{f\in A[\mathbf x,x_n] \ \vert \ \exists t\geq 0 : h^tf\in I\}$ (see \citep{BS}).
\item \label{item:remgeneric} 
Recall that the smallest variable is generic for a homogeneous polynomial ideal $I\subset A[\mathbf x,x_n]$ generated by a monic Gr\"obner basis if and only if it is generic for $\In(I)$. Indeed, the initial 
term  with respect to degrevlex of a homogeneous  polynomial $f$ is divisible by $x_n^r$ if and only if $f$ is divisible by $x_n^r$.
\item \label{item:estensioni e contrazioni}
If $L$ is a homogeneous ideal of $K[\mathbf x,x_n]$ then 
$L A[\mathbf x,x_n]=L\otimes_K A$, \ $L = L A[\mathbf x,x_n] \cap~K[\mathbf x,x_n]$, \ $(L A[\mathbf x,x_n])^\sat = L^\sat A[\mathbf x,x_n]$, and if $I$ is a homogeneous ideal of $A[\mathbf x,x_n]$ then $I^\sat \cap K[\mathbf x,x_n] = (I \cap K[\mathbf x,x_n])^\sat$.
\end{enumerate}  
\end{remark}


\section{Background I: Gr\"obner functor and functor of $x_n$-liftings}
\label{sec:Groebner strata and $x_n$-liftings}

In this section, referring to \citep{LR,LR2} and to \citep{BCGR}, we collect some known information about Gr\"obner functor and functor of $x_n$-liftings.
Both these functors are subfunctors of a Hilbert functor, for which we refer to \citep{GroHilb,Nit}. We only recall that the Hilbert functor $\underline{\mathrm{Hilb}}^n$ associates to a locally Noetherian $K$-scheme $S$ the set $
\underline{\mathrm{Hilb}}^n(S) = \Bigl\{W \subset \mathbb P_S^n \ \vert \ W \text{ is flat over }S\Bigr\}
$ of flat families of subschemes of $\mathbb P^n_S=\mathbb P^n_K\times_{\mathrm{Spec}(K)} S$ parameterized by $S$, 
and to any morphism $\phi: T \rightarrow S$ of locally Noetherian $K$-schemes the map
$$\begin{array}{lccc}\underline{\mathrm{Hilb}}^n(\phi): &\underline{\mathrm{Hilb}}^n(S) &\longrightarrow &\underline{\mathrm{Hilb}}^n(T)\\
& W &\longmapsto &W \times_S T.
\end{array}
$$
Grothendieck shows that the functor  $\underline{\mathrm{Hilb}}^n$ is representable by a locally Noetherian scheme $\mathrm{Hilb}^n$, called Hilbert scheme (see \citep{GroHilb}).
Given the subfunctor 
\[
\underline{\mathrm{Hilb}}^n_{p(t)}(S):=\Bigl\{W \subset \mathbb P_S^n \ \vert \ \ W \text{ is flat over }S\text{ and has fibers with Hilbert polynomial } p(t)\Bigr\},
\]
then $\underline{\mathrm{Hilb}}^n$ decomposes as co-product:
\begin{equation}\label{eq:co-product}
\underline{\mathrm{Hilb}}^n=\coprod_{p(t)\text{ admissible for schemes in } \mathbb P_K^n}\underline{\mathrm{Hilb}}^n_{p(t)}.
\end{equation}
For every admissible polynomial $p(t)$ of $\mathbb P_K^n$, $\underline{\mathrm{Hilb}}^n_{p(t)}$ is represented by a projective scheme $\mathrm{Hilb}_{p(t)}^n$.
The fact that $\mathrm{Hilb}^n$ and $\mathrm{Hilb}_{p(t)}^n$ are locally Noetherian allows to consider the restriction of the functors $\underline{\mathrm{Hilb}}^n$ and $\underline{\mathrm{Hilb}}^n_{p(t)}$ to the category of Noetherian $K$-algebras (e.g.~\cite[Proposition VI-2 and Exercise VI-3]{EH}). 
Every $K$-point of a Hilbert scheme is identified with the saturated ideal $I\subseteq K[x_0,\dots,x_n]$ that defines the corresponding fiber in $\mathbb P^n_K$. 

Since in this paper we only consider the degrevlex order, we now recall the notion of Gr\"obner functor in this particular setting.

\begin{definition}\label{def:GrStratum} {\rm \cite[Section 5 and  Theorem 5.3]{LR2}} 
Given a monomial ideal $J\subset K[\mathbf x,x_n]$, the {\em Gr\"obner functor} $\underline{\St}_J: \text{Noeth-}K\text{-Alg} \rightarrow \text{Sets}$ associates to any Noetherian $K$-algebra $A$ the set $\underline{\St}_J(A):=\{I\subset A[\mathbf x,x_n] : \mathrm{in}(I)=J\otimes_K A\}$ and to any $K$-algebra morphism $\phi: A\rightarrow B$ the function $\underline{\St}_J(\phi): \underline{\St}_J(A) \rightarrow \underline{\St}_J(B)$ such that the image of an ideal $I$ is $I\otimes_A B$. The affine scheme $\mathrm{St}_J$ representing the Gr\"obner functor is called {\em Gr\"obner stratum}. 
\end{definition}

In order to briefly recall the construction of a Gr\"obner stratum $\mathrm{St}_J$, take the set of polynomials
\begin{equation}\label{JbaseC}
G:= \left\{ f_\alpha=x^\alpha +  \sum_{{red}x^\alpha\succ x^\gamma\in\mathcal N(J)_{\vert \alpha\vert}} C_{\alpha\gamma} x^\gamma \ : \ 
\In(f_\alpha) = x^\alpha\in B_J \right\} \subset K[C_J][\mathbf x,x_n]
\end{equation}
where $C_J$ denotes the set of the new variables $C_{\alpha \gamma}$.
Let $\mathfrak{a}_J$ be the ideal in $K[C_J]$ generated by the coefficients  (w.r.t.~variables $\mathbf x,x_n$) of the terms 
 in $J$-reduced forms $\overline{S(f_\alpha,f_{\beta})}$ of the $S$-polynomials $S(f_\alpha,f_{\beta})$ modulo $(G)$. Due to \cite[Proposition 3.5]{LR} the ideal $\mathfrak{a}_J$ depends only on $J$ and on the given term order, which here is supposed to be the degrevlex. Hence, the ideal $\mathfrak{a}_J$ defines the affine scheme $\mathrm{St}_J=\mathrm{Spec}(K[C_J]/\mathfrak{a}_J)$.

\begin{theorem} \label{th:main features} \rm{(\cite[Lemma 5.2]{LR2},  \cite[Theorem 2.2]{BCGR})} 
Let $J\subset \Aox$ be a monomial ideal and $p(t)$ the Hilbert polynomial of $\Aox/J$. With the degrevlex term order, 
\begin{itemize}
\item[(i)] $\underline{\mathrm{St}}_{J}$ is a Zariski sheaf. 
\item[(ii)] If the terms in $B_J$ are not divisible by $x_n$, then 
$\mathrm{St}_J\cong \mathrm{St}_{J_{\geq m}}, \text{ for every integer } m$, 
and $\underline{\mathrm{St}}_J$ is a locally closed subfunctor of the Hilbert functor $\underline{\mathrm{Hilb}}^n_{p(t)}$. 
\end{itemize}
\end{theorem}

\begin{definition}\label{def:lifting} \citep{GGR,Roitman}
A homogeneous ideal $I$ of $A[\mathbf x,x_n]$ is called a {\em $x_n$-lifting} of a homogeneous ideal $H$ of $K[\mathbf x]$ if 
\begin{enumerate}
\item[(a)] the indeterminate $x_n$ is a non-zero divisor in $\Aox/I$;
\item[(b)] $(I,x_n)/(x_n)\simeq H A[\mathbf x]$ under the canonical isomorphism 
 $\Aox/(x_n)\simeq A[\mathbf x]$; 

or, equivalently,
\item[(b$'$)] $\{ g(x_0,x_1,\ldots,x_{n-1},0) : g\in I\} = H A[\mathbf x]$.
\end{enumerate}
\end{definition}

\begin{theorem} {\rm(\cite[Theorem 2.5]{CFRo}, \cite[Proposition 6.2.6]{KR2},  \cite[Theorem 3.2 and Corollary 3.3]{BCGR})}\label{FM}
Let $H$ be a homogeneous ideal of $K[\mathbf x]$. A homogeneous ideal $I$ of $\Aox$ is a $x_n$-lifting of $H$ if and only if the reduced Gr\"obner basis of $I$ is of type $\lbrace f_\alpha+g_\alpha\rbrace_\alpha$, 
where $\lbrace f_\alpha\rbrace_\alpha$ is the reduced Gr\"obner basis of $H$ and $g_\alpha \in (x_n)\Aox$.
If $I$ is a $x_n$-lifting of $H$, then   $\mathrm{in}(I)$ is generated by the same terms as $\mathrm{in}(H)$.
\end{theorem}

\begin{definition}\label{def:xn-liftings} \cite[Definition 3.4]{BCGR}
The {\em functor} \ $\underline{\mathrm{L_{H}}}: \underline{\text{Noeth-}K\text{-Alg}}\rightarrow \underline{\mathrm{Sets}}$ {\em of $x_n$-liftings} of a homogeneous ideal $H$ of $K[\mathbf x]$
associates to every Noetherian $K$-algebra $A$ the set $
\underline{\mathrm{L_{H}}}(A)=\lbrace I\subseteq \Aox : I \text{ is a $x_n$-lifting of }H \rbrace$ 
and to every morphism of $K$-algebras $\phi:A \rightarrow B$ the map
\[\begin{array}{rcl}
\underline{\mathrm{L_{H}}}(\phi): \underline{\mathrm{L_{H}}}(A)&\rightarrow& \underline{\mathrm{L_{H}}}(B)\\
 I&\mapsto& I\otimes_A B.
\end{array}\]
\end{definition}
With the notation of Definition \ref{def:xn-liftings}, let $J:=\In(H)A[\mathbf x,x_n]$ and $p(t):=p_{A[\mathbf x,x_n]/J}$. The functor $\underline{\mathrm{L_{H}}}$ is a closed subfunctor of $\underline{\St}_J$ represented by a closed affine subscheme $\mathrm{L_H}$ of $\mathrm{St}_J$ and, hence, a locally closed subscheme of the Hilbert scheme $\mathrm{Hilb}_{p(t)}^n$ thanks to Theorem \ref{th:main features} \cite[Theorem 4.3 and Proposition 6.1]{BCGR}. Moreover, $\underline{\mathrm{L_H}}$ is a Zariski sheaf.


\section{Liftings of projective schemes}
\label{sec:Liftings of projective schemes}

\begin{definition}\label{def:geometric lifting}
Let $\plane=\mathbb P^{n-1}_K$ be the hyperplane of $\mathbb P^n_K$ defined by  the ideal $(x_n)$, $Y$ be a closed subscheme of $ \mathbb P^{n-1}_K$ with Hilbert polynomial $p_Y(t)$   and $A$ be a $K$-algebra. A {\em lifting of $Y$ over $A$}   is a closed subscheme $W\in \underline{\mathrm{Hilb}}^n(A)$ of $\mathbb P^n_A=\mathbb P^n_K\times_{\mathrm{Spec}(K)} \mathrm{Spec}(A)$, such that:
\begin{itemize}
\item[(i)] $\Delta p_W(t):=p_W(t)-p_W(t-1)=p_Y(t)$; 
\item[(ii)] $W\cap (\plane \times_{\mathrm{Spec}(K)} \mathrm{Spec}(A)) = Y\times_{\mathrm{Spec}(K)} \mathrm{Spec}(A)$. 
\end{itemize}
\end{definition}

Observe that in Definition \ref{def:geometric lifting} we assume that the scheme $Y$ is contained in the hyperplane defined by the ideal $(x_n)$. This assumption is not restrictive because, given a linear form $h\in K[\mathbf x,x_n]$, we can always replace $h$ by the smallest variable $x_n$ thanks to a suitable (deterministic) change of coordinates. 

\begin{definition}\label{def:geometric lifting per gli ideali}
Let $I'$ be a homogeneous saturated ideal of $K[\mathbf x]$. A homogeneous saturated ideal $I$ of $A[\mathbf x,x_n]$ is called a {\em lifting of $I'$} if the following conditions are satisfied:
\begin{enumerate}
\item[(a)] the indeterminate $x_n$ is generic for $I$;
\item[(b)] $\Bigl((I,x_n)/(x_n)\Bigr)^{\sat}\simeq I' A[\mathbf x]$ under the canonical isomorphism 
 $\Aox/(x_n)\simeq A[\mathbf x]$; 

or, equivalently,
\item[(b$'$)] $(\{g(x_0,x_1,\ldots,x_{n-1},0) : g\in I\})^{\sat} = I' A[\mathbf x]$.
\end{enumerate}
\end{definition}

\begin{proposition}\label{prop:h generico}
Let $Y$ be a closed subscheme of $\mathbb P^{n-1}_K$ defined by the homogeneous saturated ideal $I'\subset K[\mathbf x]$ and $W\in \underline{\mathrm{Hilb}}^n(A)$ be a closed subscheme of $\mathbb P^n_A$ defined by the homogeneous saturated ideal $I\subset  A[\mathbf x,x_n]$. Then, $W$ is a lifting of $Y$ if and only if $I$ is a lifting of $I'$.
\end{proposition}

\begin{proof}
By condition (ii) of Definition \ref{def:geometric lifting}, we have $\Bigl(\frac{(I,x_n)}{(x_n)}\Bigr)^{\sat}=I' A[\mathbf x]=I'\otimes_K A$, in particular the quotient $A[\mathbf x,x_n]/(I,x_n)$ has Hilbert polynomial $p_Y(t)$. Thus, by the following short exact sequence
\begin{equation}\label{short exact sequence zero-divisor}
0 \rightarrow (\Aox/(I:x_n))_{t-1} \xrightarrow{\cdot x_n} (\Aox/I)_t \rightarrow (\Aox/(I,x_n))_t \rightarrow 0
\end{equation}
condition (i) of Definition \ref{def:geometric lifting}, i.e.~$\Delta p_W(t)= p_Y(t)$, implies that the quotient $A[\mathbf x,x_n]/(I:x_n)$ has the same Hilbert polynomial $p_W(t)$ as $A[\mathbf x,x_n]/I$. Hence, we obtain $(I:x_n)^{\sat}=I^{\sat}=I$ because $(I:x_n)\supseteq I$. In conclusion, we have $I^{\sat}\supseteq(I:x_n)^{\sat}\supseteq (I:x_n) \supseteq I=I^{\sat}$, 
namely $(I:x_n)=I$, which is possible only if $x_n$ is generic for $I$.

Conversely, if $I$ is a lifting of $I'$ then it is quite immediate that $W$ is a lifting of $Y$. Indeed, if $x_n$ is generic for $I$, then $(I:x_n)=I=I^\sat$. Hence, $A[\mathbf x, x_n]/(I:x_n)$ and $ A[\mathbf x, x_n]/I$ have the same Hilbert polynomial. From \eqref{short exact sequence zero-divisor}, we obtain $\Delta p(t)=p_{Y}(t)$ and condition (ii) of Definition~\ref{def:geometric lifting} also follows.
\end{proof}

With the notation we have already introduced in Definition \ref{def:geometric lifting}, we consider a closed subscheme $Y\subset \mathcal H \subset \mathbb P^{n}_K$, where $\mathcal H$ is defined by the ideal $(x_n)$. If $W\subset \mathbb P^n_A$ is a lifting of $Y$, then $\deg W=\deg Y$ and $\dim W = \dim Y +1$, so there are natural restrictions on the Hilbert polynomial $p(t)$ of $W$ because $\Delta p(t)$ must be the Hilbert polynomial $p_Y(t)$ of $Y$. Hence, the non-constant part of the Hilbert polynomial of $W$ is determined by the Hilbert polynomial of $Y$. However, in general 
there are no limits on the constant term of the Hilbert polynomial of $W$, even  
if we only consider liftings without zero-dimensional components, as shown by the following example.

\begin{example}\label{ex:primoes}For every positive integer $k$, consider the double line $W_k \subset \mathbb P^3_K$ defined by the ideal $I=(x_0^2,x_0x_1,x_1^2,x_0x_2^k-x_1x_3^k)\subseteq K[x_0,x_1,x_2,x_3]$. The Hilbert polynomial of $W_k$ is $p(t)=2t+k+1$ and $W_k$ is a lifting of the double point $Y\subset \mathbb P^2_K$ defined by the ideal $I'=(x_0, x_1^2)\subseteq K[x_0,x_1,x_2]$. In conclusion, for every positive integer $k$, we find a lifting of $Y$ with Hilbert polynomial $2z+k+1$ and without zero-dimensional components.
\end{example}

We conclude this section highlighting a geometric feature of liftings. 
 
\begin{proposition} \label{prop:smooth points}
Let $W$ be a lifting of a scheme $Y$. If $Y$ is smooth on a point $P$ then also $W$ is smooth on $P$. 
\end{proposition}

\begin{proof}
By definition of lifting, $Y$ is a Cartier divisor in $W$. Then the dimension of the Zariski tangent space of a point $y$ in $Y$ is not lower than the dimension of the Zariski tangent space of the point $y$ in $W$ minus $1$. Indeed, if $\mathfrak m$ is the local ring at $y$ in $Y$ and $M$ is the local ring at $y$ in $W$, we have $\dim_K \frac{\mathfrak m}{\mathfrak m^2} \geq \dim_K \frac{M}{M^2}-1$. Hence, if we had $\dim_K \frac{M}{M^2} > \dim(W)$ that we would obtain $\dim_K \frac{\mathfrak m}{\mathfrak m^2}\geq \dim_K \frac{M}{M^2}-1 > \dim(W)-1=\dim(Y)$.
\end{proof}


\section{Construction of liftings of projective schemes}
\label{sec:Construction}

In this section, we obtain a constructive characterization of liftings of projective schemes by investigating relations between the notion of lifting of a saturated homogeneous ideal in $K[\mathbf x]$ (Definition \ref{def:geometric lifting per gli ideali}) and that of $x_n$-lifting of a homogeneous ideal in $K[\mathbf x]$ (Definition \ref{def:lifting}). 

In general a lifting is not a $x_n$-lifting, as shown by the following easy example. Nevertheless, we will show how to recover every lifting of a given saturated ideal by constructing the $x_n$-liftings of suitable families of ideals.

\begin{example}
The ideal 
$I=(x_0^2+x_3^2,x_0x_1,x_0x_2,x_1^2,x_1x_2)\subset K[x_0,x_1,x_2,x_3]$ is a lifting of $I'=(x_0,x_1)\subseteq K[x_0,x_1,x_2]$ but is not a $x_3$-lifting of $I'$, as one can easily verify using Theorem \ref{FM}.
\end{example}

\begin{lemma} \label{lemma:sezione}
Let $I'\subset K[\mathbf x]$ be a saturated ideal. If $I\subset A[\mathbf x,x_n]$ is a lifting of $I'$ then $I'=\Bigl(\frac{(I,x_n)}{(x_n)}\Bigr)^{\sat}\cap K[\mathbf x]=\Bigl(\frac{(I,x_n)}{(x_n)} \cap K[\mathbf x]\Bigr)^{\sat}$. 
\end{lemma}

\begin{proof}
Definition \ref{def:geometric lifting per gli ideali} 
of lifting immediately implies the thesis by Remark \ref{rem:generic and degrevlex}\eqref{item:estensioni e contrazioni}. 
\end{proof}

\begin{proposition}\label{prop:geometric lifting}
Let $I'\subseteq K[\mathbf x]$ be a homogeneous saturated ideal. A homogeneous ideal $I\subseteq A[\mathbf x,x_n]$ is a lifting of $I'$ if and only if there exists a homogeneous ideal $H\subseteq K[\mathbf x]$ such that $H^{\sat}=I'$ and $I$ is a $x_n$-lifting of $H$.
\end{proposition}

\begin{proof}
First assume that $I$ is a lifting of $I'$ and take $H:=\frac{(I,x_n)}{(x_n)}\cap K[\mathbf x]$. Then, we have $H^{\sat}=I'$ by Lemma \ref{lemma:sezione}. Furthermore, it is immediate that $I$ is a $x_n$-lifting of $H$,  because $HA[\mathbf x]=\frac{(I,x_n)}{(x_n)}$ by Remark \ref{rem:generic and degrevlex}\eqref{item:estensioni e contrazioni} and $x_n$ is a non-zero divisor in $\Aox/I$ by definition.

Conversely, let $I\subset A[\mathbf x,x_n]$ be an ideal which is a $x_n$-lifting of a homogeneous ideal $H\subseteq K[\mathbf x]$ such that $H^{\sat}=I'$. Then, $I$ is a lifting of $I'$, because $x_n$ is generic for $I$, so $I$ is saturated, and $\frac{(I,x_n)}{(x_n)} \simeq H\otimes_K A$ implies $((I,x_n)/(x_n))^\sat=I'A[\mathbf x,x_n]$, by Remark \ref{rem:generic and degrevlex}\eqref{item:estensioni e contrazioni}.
\end{proof}

\begin{theorem} \label{th:parametrizzazione}
Let $I'\subseteq A[\mathbf x]$ be a homogeneous saturated ideal and $I\subseteq A[\mathbf x,x_n]$ a lifting of $I'$. The locus of liftings of $I'$ in $\underline{\St}_{\In(I)}(A)$ is parameterized by an affine scheme obtained by linear sections of ${\St}_{\In(I)}$.
\end{theorem}

\begin{proof}
Let $J':=\In(I')\subset K[\mathbf x]$ and $J:=\In(I)\subset A[\mathbf x,x_n]$. 
By Proposition \ref{prop:geometric lifting}, $I$ is a $x_n$-lifting of a homogeneous ideal $H\subset K[\mathbf x]$ with $H^{\sat}=I'$. Hence, by Theorem \ref{FM}, $\In(H)$ has the same generators as $J$, so $\frac{(J,x_n)}{(x_n)}=\In(H)\otimes_K A$ and $\In(H)=J\cap K[\mathbf x]$. 

Let $\mathfrak a_J \subset A[C_J]$ be the defining ideal of the Gr\"obner stratum  $\mathrm{St}_{J}$ and $\mathfrak a_{J\cap K[\mathbf x]} \subset K[C_{J\cap K[\mathbf x]}]$ the defining ideal of the Gr\"obner stratum $\mathrm{St}_{J\cap K[\mathbf x]}$, where $C_{J\cap K[\mathbf x]}\subseteq C_J$.

The ideal $H\subset K[\mathbf x]$ is characterized by the following two conditions. The first condition is that $H$ belongs to the family $\underline{\St}_{J\cap K[\mathbf x]}(K)$, so the reduced Gr\"obner basis of $H$ consists of polynomials of the following type
\begin{equation}\label{base di H}
f_\beta = x^\beta +\sum_{x^\beta >x^\gamma \in \mathcal N(\In(H))_{\vert \beta\vert}} C_{\beta\gamma} x^\gamma, \quad f_\beta  \in K[C_{J\cap K[\mathbf x]}][\mathbf x],
\end{equation}
for every term $x^\beta$ minimal generator of $J$. The second condition is that the polynomials $f_\beta$ belong to $I'$, because $H$ is contained in $I'$. This second condition implies that the saturation of $H$ is $I'$ because $K[\mathbf x]/I'$ and $K[\mathbf x]/H$ have the same Hilbert polynomial, by the first condition.

By construction, $\In(H)$ is contained in $J'$. Hence, we have $\mathcal N(J')\subseteq\mathcal N(\In(H))$ and can obtain a $J$-reduced form $\overline{f_\beta}$ modulo $I'$ of the polynomial $f_\beta$ using the reduced Gr\"obner basis of $I'$. Imposing that $\overline{f_\beta}$ is zero we obtain that $H$ is contained in $I'$ and collect some new constraints on the coefficients in the parameters $C_{\beta\gamma}$ in $C_{J\cap K[\mathbf x]}$. Let $\mathfrak b_{J\cap K[\mathbf x]} \subset K[C_{J\cap K[\mathbf x]}]$ be the ideal generated by these constraints, for every $x^\beta\in B_J$. The ideal $\mathfrak a_{J\cap K[\mathbf x]} + \mathfrak b_{J\cap K[\mathbf x]} \subset K[C_{J\cap K[\mathbf x]}]$, hence the affine scheme $\Spec\left(\frac{ K[C_{J\cap K[\mathbf x]}]}{\mathfrak a_{J\cap K[\mathbf x]} + \mathfrak b_{J\cap K[\mathbf x]}}\right)$, parameterizes the locus in $\underline{\St}_{J\cap K[\mathbf x]}(K)$ of all the ideals $H\subset K[\mathbf x]$ such that $H^\sat=I'$. 

Finally, we apply Theorem \ref{FM} and hence consider 
\begin{equation}\label{base di I}
g_\beta :=f_\beta + \sum_{x^\delta \in \mathcal N(J)_{\vert\beta\vert-1}} C_{\beta\delta} x_n x^\delta, \quad g_\beta  \in A[C_J][\mathbf x,x_n],
\end{equation}
for every term $x^\beta$ minimal generator of $J$. The set $\{g_\beta\}_{x^\beta \in B_J}$ is a Gr\"obner basis with initial ideal $J$ modulo the ideal $\mathfrak a_J\subset A[C_J]$ which defines the Gr\"obner stratum $\mathrm{St}_{J}$.

We now observe that if the set of polynomials $g_\beta$ is a Gr\"obner basis then also the set of polynomials $f_\beta$ is a Gr\"obner basis, due to the hypothesis on $J$ and $\In(H)$. This fact means that the ideal $\mathfrak a_{J\cap K[\mathbf x]} A[C_J]$ is contained in $\mathfrak a_J$. Then, the ideal $\mathfrak a_J+\mathfrak b_{J\cap K[\mathbf x]}  A[C_J]$, hence the affine scheme $\Spec\left(\frac{A[C_J]}{\mathfrak a_J+\mathfrak b_{J\cap K[\mathbf x]}  A[C_J]}\right)$, parameterizes the locus of the liftings of $I'$ in the family $\underline{\St}_{J}(A)$. 

It remains to show that the constraints we obtain by rewriting the polynomials $f_\beta$ of \eqref{base di H} by the reduced Gr\"obner basis of $I'$ are linear, i.e.~the ideal $\mathfrak b_{J\cap K[\mathbf x]}$ has linear generators. This fact is immediate, because the coefficients of the polynomials of the reduced Gr\"obner basis of $I'$ belong to the field $K$.
\end{proof}

Theorem \ref{th:parametrizzazione} describes the locus of liftings of a homogeneous saturated ideal $I'\subset K[\mathbf x]$ in a given Gr\"obner stratum. In the further hypothesis that the variable $x_{n-1}$ is generic for $I'$,  we can recognize what Gr\"obner strata are candidate to contain these liftings.

\begin{theorem}\label{th:dove}
Let $I'\subseteq K[\mathbf x]$ be a homogeneous saturated ideal. If $x_{n-1}$ is generic for $I'$, then the liftings of $I'$ belong to the Gr\"obner stratum over a monomial lifting $J\subset A[\mathbf x,x_n]$ of \ $\mathrm{in}(I')$.
\end{theorem}

\begin{proof}
It is enough to prove that if $I\subset A[\mathbf x,x_n]$ is a lifting of $I'$ and $x_{n-1}$ is generic for $I'$, then $\Bigl(\frac{(\In(I),x_n)}{(x_n)}\Bigr)^{\sat}=\In(I')\otimes_K A$. Hence, we will conclude by applying Theorem \ref{th:parametrizzazione} and observing that if $x_n$ is generic for $I$ then it is a non-zero divisor for $\In(I)$.

Let $J':=\In(I')$ and $J:=\In(I)$. By definition of lifting we have $\Bigl(\frac{(I,x_n)}{(x_n)}\Bigr)^{\sat}=I'\otimes_K A$. Hence, there exists an integer $s\geq 0$ such that
$\frac{(I_{\geq s},x_n)}{(x_n)}=\Bigl(\frac{(I,x_n)}{(x_n)}\Bigr)_{\geq s}=I'_{\geq s}\otimes_K A$. 
By \cite[Lemma 2.2]{BS}, we have $\In(I_{\geq s},x_n)=(\In(I_{\geq s}),x_n)$ and obtain
\begin{equation}\label{eq:saturation}
\frac{(J_{\geq s},x_n)}{(x_n)}={J'}_{\geq s}\otimes_K A
\end{equation}
because $\In(I_{\geq s})=\In(I)_{\geq s}= J_{\geq s}$ and $\In(I'_{\geq s})=\In(I')_{\geq s}= J'_{\geq s}$.
Now, it is enough to recall that $J'$ is saturated because $I'$ is saturated and $x_{n-1}$ is not a zero-divisor in $K[\mathbf x]/I'$.
\end{proof}

\begin{remark} 
The condition that $x_{n-1}$ is generic for $I'$ is not restrictive because it can always be obtained up to a suitable change of variables. On the other hand, the result of Theorem \ref{th:dove} does not hold without this hypothesis: for example, for the saturated ideal $I'=(x_0^2,x_1x_0+x_2^2)\subset  K[x_0,x_1,x_2]$ we obtain $\In(I')=(x_0^2, x_1x_0,x_2^2x_0,x_2^4)$, that is not saturated.
\end{remark}

Thanks to Theorems \ref{th:parametrizzazione} and \ref{th:dove}, we know {\em where} the liftings of a given saturated polynomial ideal $I'\subseteq A[\mathbf x]$ are located and {\em how} they can be constructed, obtaining Algorithm \ref{alg:computation in GS}. 

\begin{algorithm}[!ht]
\caption{\label{alg:computation in GS} Algorithm for computing parameter schemes for the liftings of a saturated homogeneous ideal $I'\subset K[\mathbf x]$ over a Noetherian $K$-algebra $A$ in $\mathrm{Hilb}_{p(t)}^n$ by means of Gr\"obner strata.}
\begin{algorithmic}[1]
\STATE $\textsc{LiftingGS}\big(I',p(t)\big)$
\REQUIRE $I'\subset K[\mathbf x]$ a saturated polynomial ideal such that $x_{n-1}$ is generic for $I'$.
\REQUIRE  $p(t)$ a Hilbert polynomial such that $\Delta p(t)= p_Y(t)$, where $p_Y(t)$ is the Hilbert polynomial of the scheme $Y$ defined by $I'$.
\ENSURE  A set $\mathfrak B$ containing parameter schemes for the liftings of $I'$ in the Gr\"obner stratum over $J$, for every monomial lifting $J\subset \Aox$ of $\In(I')$ with Hilbert polynomial $p(t)$.
\STATE $\mathcal L:=\{ J \subset \Aox \ \vert \ J \text{ monomial lifting of } \In(I') \text{ with Hilbert polynomial } p(t)\}$; 
\STATE $\mathfrak B=\emptyset$;
\FOR{$J \in \mathcal L$}
\STATE let $\mathfrak a_J \subset A[C_J]$ be the defining ideal of the Gr\"obner stratum $\mathrm{St}_{J}$;
\STATE $\mathfrak b_{J\cap K[\mathbf x]}:=(0)$;
\FOR{$x^\beta\in B_J$}
\STATE construct the polynomial $f_\beta$ as in \eqref{base di H};
\STATE $\mathfrak b_{J\cap K[\mathbf x]}:= \mathfrak b_{J\cap K[\mathbf x]}+(\mathrm{coeff}(\overline{f_\beta}))$, where $\overline{f_\beta}$ is the $J$-normal form of $f_\beta$ modulo $I'$; 
\ENDFOR
\STATE $\mathfrak B:=\mathfrak B\cup \{\mathfrak a_J+\mathfrak b_{J\cap K[\mathbf x]}\Aox\}$;
\ENDFOR
\end{algorithmic}
\end{algorithm}


\section{Functors of liftings}

From Theorem \ref{th:parametrizzazione} we obtain the following generalization of \cite[Corollary~3.3]{BCGR} 
in which, if $\phi:A\rightarrow B$ is a $K$-algebra morphism, we also denote by $\phi$ the natural extension 
of $\phi$ to $\Aox$. Recall that the image under $\phi$ of every ideal $I$ in $\Aox$ generates the extension  $I^e = IB[\mathbf x,x_n]=I \otimes_A B$ (see \citep{BGS91}). 

\begin{proposition} \label{prop:funtore}
Let $I'\subset K[\mathbf x]$ be a saturated ideal and $\phi:A\rightarrow B$ a $K$-algebra morphism. For every lifting $I \subset \Aox$ of $I'$ over $A$ the ideal $I\otimes_A B$ is a lifting of $I'$ over $B$, with the same Hilbert polynomial as $I$. 
\end{proposition}

\begin{proof} Let $G_I$ be the reduced monic Gr\"obner basis of $I$ and $J=\In(I)$. Then, $\phi(G_I)$ is a monic Gr\"obner basis of $I\otimes_A B$ with the same  initial 
terms of $G_I$ because $G_I$ is monic. Let $H\subset K[\mathbf x]$ be the ideal such that $I$ is a $x_n$-lifting of $H$ and $H^\sat=I'$. Let $G_H=\{f_\beta\}_{x^\beta\in B_J}$ be the reduced (monic) Gr\"obner basis of $H$. By Theorems \ref{th:parametrizzazione} and \ref{FM} we have that $G_I$ is of the following type
\[
G_I=\lbrace f_\beta + \sum_{x^\delta \in \mathcal N(J)_{\vert \beta\vert-1}} {c}_{\beta\delta} x_n x^\delta \rbrace_\beta, \quad c_{\beta\delta} \in A.
\]
The ideal $I\otimes_A B$ is then generated by $\phi(G_I)=\lbrace f_\beta - \sum_{x^\delta \in \mathcal N(J)_{\vert \beta\vert-1}} \phi({c}_{\beta\delta}) x_n x^\delta \rbrace_{x^\beta\in B_J}$, which is still a reduced Gr\"obner basis  because the polynomials of $G_I$ are monic. Hence, $I\otimes_A B$ is a $x_n$-lifting of $H$ by Theorem \ref{FM} 
and a lifting of $I'$ by Proposition  \ref{prop:geometric lifting}.

For the statement concerning 
the Hilbert polynomial, it is sufficient to observe that $I$ and $I\otimes_K B$ have the same initial ideal, hence the same Hilbert polynomial.
\end{proof}

Given a scheme $Y\subseteq \mathbb P^{n-1}_K$, thanks to Proposition \ref{prop:funtore} we can now easily define some functors concerning the liftings of $Y$. 

\begin{definition}\label{def:funtore lifting}
Let $Y=\Proj(K[\mathbf x]/I')$ be a closed subscheme of $\mathbb P^{n-1}_K$ with Hilbert polynomial $p_Y(t)$ and $p(t)$ be a Hilbert polynomial such that $\Delta p(t)=p_Y(t)$. 
\begin{itemize}
\item[(a)] The {\em functor of liftings of $Y$},  $\underline{\mathrm{L}}_{Y}:{\text{Noeth-}K\text{-Alg}}\rightarrow {\mathrm{Sets}}$, 
associates to every Noetherian $K$-algebra $A$ the set 
$$\underline{\mathrm{L}}_{Y}(A):=\{I \subset A[\mathbf x,x_n] : I \text{ lifting of } I' \}$$
and to every morphism of $K$-algebras $\phi: A \rightarrow B$ the map
\[\begin{array}{rcl}
\underline{\mathrm{L}}_{Y}(\phi): \underline{\mathrm{L}}_{Y}(A)&\rightarrow& \underline{\mathrm{L}}_{Y}(B)\\
 I&\mapsto& I\otimes_A B.
\end{array}\]

\item[(b)] The {\em functor of liftings of $Y$ with Hilbert polynomial $p(t)$}, $\underline{\mathrm{L}}_Y^{p(t)}: {\text{Noeth-}K\text{-Alg}}\rightarrow {\mathrm{Sets}}$, associates to every Noetherian $K$-algebra $A$ the set 
$$\underline{\mathrm{L}}_{Y}^{p(t)}(A):=\{I \subset A[\mathbf x,x_n]  : I \text{ lifting of } I' \text{ with Hilbert polynomial } p(t)\}$$
and to every morphism of $K$-algebras $\phi: A \rightarrow B$ the map
\[\begin{array}{rcl}
\underline{\mathrm{L}}_{Y}^{p(t)}(\phi): \underline{\mathrm{L}}_{Y}^{p(t)}(A)&\rightarrow& \underline{\mathrm{L}}_{Y}^{p(t)}(B)\\
 I&\mapsto& I\otimes_A B.
\end{array}\]
\end{itemize}
\end{definition}

It is immediate that $\underline{\mathrm{L}}_{Y}^{p(t)}$ is a subfunctor of $\underline{\mathrm{L}}_{Y}$, and furthermore $\underline{\mathrm{L}}_{Y}$ (resp.~$\underline{\mathrm{L}}_{Y}^{p(t)}$) is a subfunctor of $\underline{\mathrm{Hilb}}^n$ (resp.~$\underline{\mathrm{Hilb}}^n_{p(t)}$). In fact, recall that the functors $\underline{\mathrm{Hilb}}^n$ and $\underline{\mathrm{Hilb}}^n_{p(t)}$ are both representable by locally Noetherian schemes, hence it is enough to consider their restrictions to the category of Noetherian $K$-algebras.

Using the same arguments on Hilbert schemes that give \eqref{eq:co-product}, we obtain that the functor of litings of $Y$ decomposes as a co-product of the above subfunctors 
\begin{equation}\label{eq:co-product lifting}
\underline{\mathrm{L}}_{Y}=\coprod_{p(t)\text{ admissible for liftings of $Y$ in } \mathbb P_K^n} \underline{\mathrm{L}}_{Y}^{p(t)}.
\end{equation}

\begin{proposition}\label{prop:Zariski sheaf}
The functor $\underline{\mathrm{L}}_Y^{p(t)}$ is a Zariski sheaf.
\end{proposition}

\begin{proof}
Let $A$ be a Noetherian $K$-algebra and $\{U_i=\Spec(A_{a_i})\}_{i=1,\dots,r}$ be an open covering of $\Spec(A)$. This is equivalent to the fact $(a_1,\dots,a_r)=A$. Consider a set of ideals $I_i\in \underline{\mathrm{L}}_Y^{p(t)}(A_{a_i})$ such that for any pair of indexes $i\not= j$ we have
\begin{equation}\label{eq:Zariski sheaf}
I_{ij}:=I_i\otimes_{A_{a_i}} A_{a_ia_j} = I_j\otimes_{A_{a_j}} A_{a_ia_j} \in \underline{\mathrm{L}}_Y^{p(t)}(A_{a_ia_j}).
\end{equation}
We need to show that there is a unique ideal $I\in \underline{\mathrm{L}}_Y^{p(t)}(A)$ such that $I_i=I\otimes_A A_{a_i}$ for every $i$. 
 
By Proposition \ref{prop:geometric lifting}, there are $H_i$ and $H_j$ ideals in $K[\mathbf x]$ such that $H_i^\sat=H_j^\sat=I'$ and $I_i$ is a $x_n$-lifting of $H_i$, while $I_j$ is a $x_n$-lifting $H_j$. By Theorem \ref{FM} and assumption \eqref{eq:Zariski sheaf}, $H_i=H_j\subset K[\mathbf x]$, hence $I_i$ belongs to $\underline{\mathrm L_H}(A_i)$ and $I_j$ belongs to $\underline{\mathrm L_H}(A_j)$. Since $\underline{\mathrm L_H}$ is a Zariski sheaf, there is a unique $I\subset A[\mathbf x,x_n]$ such that $I$ is a $x_n$-lifting of $H$, $I\otimes_{A}A_{a_i}=I_i$ and $I\otimes_{A}A_{a_j}=I_j$. By Proposition \ref{prop:geometric lifting} we conclude that $I$ belongs to $\underline{\mathrm{L}}_Y^{p(t)}(A)$ and is the unique ideal in $A[\mathbf x,x_n]$ such that $I_i=I\otimes_A A_{a_i}$ for every $i$. 
\end{proof}

Recall that a closed subscheme in $\mathbb P^n_K$ is {\em equidimensional} if all its components have the same dimension, in particular it has no embedded components. Thus, there exists an equidimensional lifting $W$ of a subscheme $Y$ only if $Y$ is equidimensional, i.e.~the ideal $I'\subset K[\mathbf x]$ defines an equidimensional scheme in $\mathbb P^{n-1}_K$.
We say that a saturated ideal $I\subset A[\mathbf x,x_n]$ is {\em equidimensional} if defines families of equidimensional subschemes. The next well-known result 
highlights that base extension preserves the fibers on every $K$-point and, hence, their possible equidimensionality. 

\begin{lemma}\label{lemma:funtore equidimensionale}
Let $W$ be a scheme over a $K$-algebra $A$. If $\phi: A \rightarrow B$ is a morphism of $K$-algebras, then the fibers of $W\to \Spec(A)$ are isomorphic to the fibers of $W\times_{\Spec(A)} \Spec(B)\to \Spec(B)$ for every $K$-point.
\end{lemma}

\begin{proof} For the sake of completeness we give a proof of this statement. 
Let $\phi: A\to B$ be a morphism of $K$-algebras, $\phi^{*}: \Spec(B) \to \Spec(A)$ the corresponding morphism and $\Spec(K)\to \Spec(B)$  the morphism associated to a $K$-point of $\Spec(B)$ (e.g.~\cite[Chapter II, Exercise 2.7]{H}). Moreover, let $\Spec(K)\to \Spec(A)$ be the morphism associated to the $K$-point of $\Spec(A)$ obtained by composition with $\phi^{*}$ and $W\times_{\Spec(A)} \Spec(K)$ the fiber on this $K$-point. Then, we obtain $(W\times_{\Spec(A)} \Spec(B))\times_{\Spec(B)} \Spec(K)\simeq W\times_{\Spec(A)} \Spec(K)$ due to the transitivity of base extension.
\end{proof}

\begin{definition}\label{def:funtore lifting eq}
Let $Y=\Proj(K[\mathbf x]/I')$ be an equidimensional closed subscheme of $\mathbb P^{n-1}_K$ with Hilbert polynomial $p_Y(t)$ and $p(t)$ a Hilbert polynomial such that $\Delta p(t)=p_Y(t)$. Thanks to Lemma \ref{lemma:funtore equidimensionale} we define
\begin{itemize}
\item[(a)] The {\em functor of equidimensional liftings of $Y$}, denoted by $\underline{\mathrm{L}}_{Y}^e: {\text{Noeth-}K\text{-Alg}}\rightarrow {\mathrm{Sets}}$,
associates to every Noetherian $K$-algebra $A$ the set 
$$\underline{\mathrm{L}}_{Y}^e(A):=\{I \subset A[\mathbf x,x_n] : I \text{ equidimensional lifting of } I' \}$$
and to every morphism of $K$-algebras $\phi: A \rightarrow B$ the map
\[\begin{array}{rcl}
\underline{\mathrm{L}}_{Y}^e(\phi): \underline{\mathrm{L}}_{Y}^e(A)&\rightarrow& \underline{\mathrm{L}}_{Y}^e(B)\\
 I&\mapsto& I\otimes_A B.
\end{array}\]

\item[(b)] The {\em functor of equidimensional liftings of $Y$ with Hilbert polynomial $p(t)$}, denoted by $\underline{\mathrm{L}}_Y^{p(t),e}:~{\text{Noeth-}K\text{-Alg}}\rightarrow~{\mathrm{Sets}}$, associates to every Noetherian $K$-algebra $A$ the set 
$$\underline{\mathrm{L}}_{Y}^{p(t),e}(A):=\{I \subset A[\mathbf x,x_n]  : I \text{ equidimensional lifting of } I' \text{ with Hilbert polynomial } p(t)\}$$
and to every morphism of $K$-algebras $\phi: A \rightarrow B$ the map
\[\begin{array}{rcl}
\underline{\mathrm{L}}_{Y}^{p(t),e}(\phi): \underline{\mathrm{L}}_{Y}^{p(t),e}(A)&\rightarrow& \underline{\mathrm{L}}_{Y}^{p(t),e}(B)\\
 I&\mapsto& I\otimes_A B.
\end{array}\]
\end{itemize}
\end{definition}

By definition, the functor $\underline{\mathrm{L}}_Y^{e}$ (resp.~$\underline{\mathrm{L}}_Y^{p(t),e}$) is a subfunctor of $\underline{\mathrm{L}}_Y$ (resp.~of $\underline{\mathrm{L}}_Y^{p(t)}$) and we have
\begin{equation}\label{eq:co-product lifting equid}
\underline{\mathrm{L}}_{Y}^e=\coprod_{p(t)\text{ admissible for liftings of $Y$ in } \mathbb P_K^n} \underline{\mathrm{L}}_{Y}^{p(t),e}
\end{equation}
similarly to formulas \eqref{eq:co-product} and \eqref{eq:co-product lifting} for $\underline{\mathrm{Hilb}^n}$ and $\underline{\mathrm{L}}_Y$.


\section{The functor $\underline{\mathrm{L}}_Y^{p(t),e}$ is representable}

We need some preliminary results. 

\begin{proposition}\label{prop:Gro2} \cite[Proposition (2.3.4)(iii)]{Gro2}
Let $S$ be a locally Noetherian $K$-scheme, $W$ be an element of $\underline{\mathrm{Hilb}}_{p(t)}^n(S)$ and $f: W \rightarrow S$ the corresponding flat projection. For every irreducible closed subset $S'$ of $S$, every irreducible component $W'$ of $f^{-1}(S')$ is dominant on $S'$, i.e.~$f'=f_{\vert W'} : W' \rightarrow S'$ is dominant. 
\end{proposition}

If $p_Y(t)$ is the Hilbert polynomial of $Y$, we can consider the Hilbert-flag scheme $\mathcal Fl_{p_Y,p}$ (see \citep{Kleppe}) and the projections $\pi_1: \mathcal Fl_{p_Y,p} \longmapsto \mathrm{Hilb}_{p_Y(t)}^{n-1}$ and  $\pi_2: \mathcal Fl_{p_Y,p} \longmapsto \mathrm{Hilb}_{p(t)}^n$. Thus, $\pi_1^{-1}(Y)$ is the closed scheme consisting of the couples $(Y,W)$ where $W$ varies among all the closed subschemes of $\mathbb P^n_K$ containing $Y$ and with Hilbert polynomial $p(t)$. We set $\mathrm{Hilb}_{p(t),Y}^n:=\pi_2(\pi_1^{-1}(Y))$. 

\begin{proposition}\label{prop:flag}
$\mathrm{Hilb}_{p(t),Y}^n$ is a closed subscheme of $\mathrm{Hilb}_{p(t)}^n$ which represents a closed subfunctor $\underline{\mathrm{Hilb}}_{p(t),Y}^n$ of $\underline{\mathrm{Hilb}}_{p(t)}^n$. 
\end{proposition}

\begin{proof}
From the definition it follows straightforwardly that $\mathrm{Hilb}_{p(t),Y}^n:=\pi_2(\pi_1^{-1}(Y))\simeq \pi_1^{-1}(Y)$ is the closed subscheme of $\mathrm{Hilb}_{p(t)}^n$
consisting of the points in $\mathrm{Hilb}_{p(t)}^n$ corresponding to schemes containing $Y$. Thus, it represents the closed subfunctor of the Hilbert functor that associates to a scheme $S$ the set of subschemes $W$ of $\mathbb P^n_K \times S$ containing $Y\times_{\mathrm{Spec}(K)}~S$.
\end{proof}

\begin{theorem}\label{th:schema che rappresenta}
Let $Y$ be an equidimensional closed subscheme of $\mathbb P^{n-1}_K$ with Hilbert polynomial $p_Y(t)$ and $p(t)$ a Hilbert polynomial such that $\Delta p(t)=p_Y(t)$. Then, $\underline{\mathrm{L}}_Y^{p(t),e}$ is representable by a locally closed subscheme ${\mathrm{L}}_Y^{p(t),e}$ of ${\mathrm{Hilb}}_{p(t)}^n$.
\end{theorem}

\begin{proof} 
Let $S$ be a Noetherian $K$-scheme, $W$ an element of $\underline{\mathrm{Hilb}}_{p(t)}^n(S)$ and $f: W \rightarrow S$ the corresponding flat projection. The fibers of $f$ in $W$ have degree equal to $\deg(Y)$ and dimension equal to $\dim(Y)+1$ because $\Delta p(t)=p_Y(t)$. 

For every irreducible closed subset $S'$ of $S$, let $W'$ be any irreducible component of $f^{-1}(S')$. By Proposition \ref{prop:Gro2}, $W'$ is dominant on $S'$. Then, also $W'\cap (\mathcal H\times S')$ is dominant on $S'$, because for every $s'\in S'$ the fiber of $s'$ in $W'$ has dimension at least $1$ by construction, and hence $s'$ has a fiber in $W'\cap (\mathcal H\times S')$ of dimension at least $0$. Indeed, the dimension of every fiber in $W'\cap (\mathcal H\times S')$ is between $\dim(Y)+1$ and $\dim(Y)$.

Recall that the dimension of the fibers of a dominant morphism is an upper semicontinuous function, namely the subset of $S'$ whose fibers in $W\cap(\mathcal H\times S')$ have dimension less than or equal to $\dim(Y)$ is open \cite[Chapter I, section 8, Corollary 3]{Mumford}. Since this dimension cannot be strictly lower than $\dim(Y)$ by the previous argument, all the fibers of the above open subset have dimension equal to $\dim(Y)$. 

In the above situation, if we assume that $W$ belongs to $\underline{\mathrm{Hilb}}_{p(t)}^n(U_e)$, where $U_e$ is the open subscheme of $\mathrm{Hilb}_{p(t)}^n$ that parameterizes the families of equidimensional subschemes of $\mathbb P^n_K$ with Hilbert polynomial $p(t)$  \cite[Th\'{e}or\`{e}me (12.2.1)(iii)]{Gro3}, 
we can also observe that the fibers in $W$ and $\mathcal H \times_{\mathrm{Spec}(K)} S$ intersect properly, implying that the degree of the fibers in $W\cap (\mathcal H \times_{\mathrm{Spec}(K)} S)$ is less than or equal to $\deg(Y)$ because the fibers and $\mathcal H \times_{\mathrm{Spec}(K)} S$ are equidimensional (see \cite[Corollary 18.5]{Harris92} in case of varieties). 

If we also assume that $W$ belongs to $\mathrm{Hilb}_{p(t),Y}^n$, so that $W$ contains $Y$, we obtain that the fibers in $W\cap (\mathcal H \times_{\mathrm{Spec}(K)} S)$ have degree equal to $\deg(Y)$ because $Y\times_{\mathrm{Spec}(K)} S \subseteq W\cap (\mathcal H \times_{\mathrm{Spec}(K)} S)$.

We can now conclude that there is an open subset in $\mathrm{Hilb}_{p(t),Y}^n$ describing all subschemes $W$ such that $W\cap (\mathcal H \times_{\mathrm{Spec}(K)} S)=Y\times_{\mathrm{Spec}(K)} S$, namely $W$ is a lifting of $Y$. It is immediate that any equidimensional lifting of $Y$ belongs to this open subset of $\mathrm{Hilb}_{p(t),Y}^n$ and, hence, locally closed subscheme of $\mathrm{Hilb}_{p(t)}^n$ by Proposition \ref{prop:flag}. 
\end{proof}

\begin{remark} 
The locally closed subscheme ${\mathrm{L}}_Y^{p(t),e}$ of ${\mathrm{Hilb}}_{p(t)}^n$ which has been introduced in the proof of Theorem \ref{th:schema che rappresenta} completely describes the locally Cohen-Macaulay liftings of $Y$ when $Y$ is a zero-dimensional scheme. 
\end{remark}


\section{The case of a quasi-stable initial ideal}
\label{sec:quasi-stable}

Thanks to Theorem \ref{th:dove} we have that the liftings of a saturated homogeneous ideal $I'\subset K[\mathbf x]$ with $x_{n-1}$ generic belong to a Gr\"obner stratum over a monomial lifting $J$ of $J':=\mathrm{in}(I')$. In this section we prove that if $J'$ is quasi-stable then $J$ is quasi-stable too (Theorem \ref{th:lifting}). 

The assumption that $J'$ is quasi-stable is not restrictive: indeed, this can be obtained by a change of coordinates on $I'$, and this change does not effect the scheme $Y=\Proj(K[\mathbf x]/I')$ from a geometric point of view.  Quasi-stability for initial ideals will allow us to use the techniques concerning marked bases over quasi-stable ideals developed in \citep{BCR2} (see \citep{ABRS} for the more general case of free modules).

\begin{lemma}\label{lemma:ss tagliato}
Let $J\subseteq A[\mathbf x,x_n]$ be a monomial ideal.
If $J_{\geq s}$ is quasi-stable for some integer $s$, then $J$ is quasi-stable.
\end{lemma}

\begin{proof}
If $J_{\geq s}$ is quasi-stable, it is enough to check the condition of Definition \ref{def:stable} for every term $x^\alpha \in J$ with $\vert\alpha\vert <s$. For every $x_i > \min(x^\alpha)$, take $x^\alpha x_i^{s-\vert\alpha\vert} \in J_{\geq s}$. Then, there is an integer $t$ such that $\frac{x^\alpha x_i^{s-\vert\alpha\vert+t}}{\min(x^\alpha x_i^{s-\vert\alpha\vert})}$ belongs to $J_{\geq s}$, because $J_{\geq s}$ is quasi-stable. We conclude by observing that $\min(x^\alpha)=\min(x^\alpha x_i^{s-\vert\alpha\vert})$ by construction.
\end{proof}

\begin{theorem}\label{th:lifting}
If $J'\subseteq K[\mathbf x]$ is a saturated quasi-stable  ideal then a monomial lifting $J\subseteq A[\mathbf x,x_n]$ of $J'$ is quasi-stable.
\end{theorem}

\begin{proof}
Consider the ideal $L=\frac{(J,x_n)}{(x_n)}\cap K[\mathbf x]$. Since quasi-stability is a property concerning the semigroup structure of the ideal generated by the minimal monomial basis of $J$, regardless of the coefficients of the polynomial ring, it is sufficient to prove that $L$ is quasi-stable.

By the hypothesis, we have $L^\sat=J'$ and hence, if $s=\sat(L)$, then $L_{\geq s}={J'}_{\geq s}$ is quasi-stable. By Lemma \ref{lemma:ss tagliato}(i), $L$ is quasi-stable too.
\end{proof}

\begin{corollary}\label{cor:lifting}
Let $I'\subseteq K[\mathbf x]$ be a homogeneous saturated ideal and $I\subseteq A[\mathbf x,x_n]$ be a lifting of $I'$. If $\In(I')$ is quasi-stable, then $\In(I)$ is a quasi-stable lifting of $\In(I')$. 
\end{corollary}

\begin{proof}
This is a consequence of the fact that the smallest variable is generic for any quasi-stable ideal (see \cite[Prop. 4.4(ii)]{Seiler2009II}), of Remark  \ref{rem:generic and degrevlex}\eqref{item:remgeneric} and of Theorems \ref{th:lifting} and \ref{th:dove}.
\end{proof}

\begin{example}\label{ex:un primo calcolo} 
In this example we apply Theorem \ref{th:lifting}. Let $\mathbf x:=\{x_0,\dots,x_3\}$ and consider the saturated ideal $I'=(x_0^2,x_0x_1+x_1^2,x_0x_2)=(x_1^2,x_0) \cap (x_2,x_1^3,x_0x_1+x_1^2,x_0^2) \subseteq K[\mathbf x]$. 
The reduced Gr\"obner basis of $I'$ is $G'=\{x_0x_2, x_0x_1 + x_1^2, x_0^2, x_1^2x_2, x_1^3\}$, hence the initial ideal is the quasi-stable ideal $J':=\In(I')=(x_0^2,x_0x_1,x_0x_2,$ $x_1^2x_2,x_1^3)$.
The Hilbert function of $\frac{K[\mathbf x]}{I'}$ is 
\[
h_{K[\mathbf x]/I'}(0)=1, h_{K[\mathbf x]/I'}(1)= 4  \text{ and } h_{K[\mathbf x]/I'}(t)=2t+3 \text{ for every }t\geq 2, 
\]
hence $I'$ defines a curve $Y$ in $\mathbb P^3_K$. A Hilbert polynomial $p(t)$ such that $\Delta p(t)=2t+3$ must be of type $p(t)=t^2+4t+c$. 
Assume that $W\subset \mathbb P^4_K$ is a surface and is a lifting of $Y$ over a Noetherian $K$-algebra $A$. The Hilbert polynomial of $W$ is $p(t)=t^2+4t+c$ with $c\geq 0$. Indeed, since $x_4$ is not a zero-divisor on $S/I$, we have the short exact sequence
$$ 0 \rightarrow (A[\mathbf x,x_4]/I)_{t-1} \xrightarrow{\cdot x_4} (A[\mathbf x,x_4]/I)_t \rightarrow (A[\mathbf x,x_4]/(I,x_4))_t \rightarrow 0,$$
that gives $h_{A[\mathbf x,x_4]/(I,x_4)}(t) = \Delta h_{A[\mathbf x,x_4]/I}(t)$, in particular $p_{A[\mathbf x,x_4]/{(I,x_4)}}(t)=\Delta p_{A[\mathbf x,x_4]/{I}}(t)$ and $\Delta h_W(t)\geq h_Y(t)$ for every $t$. As a consequence, $h_W(t)\geq \sum_{i=0}^t h_Y(i)$, hence 
\[
h_{A[\mathbf x,x_4]/I}(0)= 1, h_{A[\mathbf x,x_4]/I}(1)= 5 \text{ and } h_{A[\mathbf x,x_4]/I}(t)\geq t^2+4t \text{ for every }t\geq 2, 
\]
and $c=0$ is the minimal possible value of the constant term $c$ for the Hilbert polynomial of $W$. 
We now investigate the cases $c=0$ and $c=1$. In the following, we compute quasi-stable ideals by the algorithm described in \citep{Bertone}. 

If $c=0$, among $56$ possible quasi-stable saturated ideals there is a unique quasi-stable ideal $J:=J'\cdot K[\mathbf x,x_4]\subset K[\mathbf x,x_4]$ such that $\left(\frac{(J,x_4)}{(x_4)}\right)^{\sat} \cap K[\mathbf x]= J'$. Hence, in this case the liftings of $I'$ belong to the family $\underline{\St}_J(A)$ and exactly are the $x_4$-liftings of $I'$ because $J$ and $J'$ share the same generators, using Theorem \ref{th:dove}.

If $c=1$,  among $176$ possible quasi-stable saturated ideals there are only the following $5$ quasi-stable ideals $J^{(i)}\subset K[\mathbf x,x_4]$ such that $\left(\frac{(J^{(i)},x_4)}{(x_4)}\right)^{\sat}\cap K[\mathbf x]=\In(I')$:
$$\begin{array}{l}
J^{(1)}=(x_2x_0, x_1x_0, x_3x_0^2, x_2x_1^2, x_1^3, x_0^3),\\
J^{(2)}=(x_2x_0, x_0^2, x_3x_1x_0, x_2x_1^2, x_1^3, x_1^2x_0),\\
J^{(3)}=(x_1x_0, x_0^2, x_3x_2x_0, x_2^2x_0, x_2x_1^2, x_1^3) 
\\
J^{(4)}=(x_2x_0, x_1x_0, x_0^2, x_2x_1^2, x_3x_1^3, x_1^4),\\
J^{(5)}=(x_2x_0, x_1x_0, x_0^2, x_1^3, x_3x_2x_1^2, x_2^2x_1^2) 
\end{array}
$$
Hence, in this case the liftings of $I'$ belong to the union of Gr\"obner strata over the above quasi-stable ideals in the Hilbert scheme $\mathrm{Hilb}_{t^2+4t+1}^4$. Now, we apply the construction described in the proof of Theorem \ref{th:parametrizzazione} to these monomial ideals and obtain the following families of Gr\"obner bases for the ideals of type $H$, modulo the defining ideal of the Gr\"obner stratum over $J^{(k)}$:

$J^{(1)}$: $G^{(1)}=\{x_2x_0, x_0x_1+x_1^2, x_3x_0^2, x_2x_1^2, x_1^3, x_0^3 \}$,

$J^{(2)}$: $G^{(2)}(a)=\{x_2x_0, 
x_0^2 + ax_0x_1 + ax_1^2, 
x_0x_1x_3+ x_1^2x_3,
x_2x_1^2,
x_1^3, 
x_1^2x_0\}$ where $a\in A$,

$J^{(3)}$: $G^{(3)}(b,c)=\{ x_0x_1+ x_1^2 + bx_0x_2,  
x_0^2 + cx_0x_2, 
x_3x_2x_0,
x_0x_2^2, 
x_2x_1^2,
x_1^3\}$ where $b,c\in A$,

$J^{(4)}$: $G^{(4)}=\{x_2x_0, 
x_0x_1+x_1^2,
x_0^2,
x_2x_1^2,
x_3x_1^3,
x_1^4\}$,

$J^{(5)}$:  $G^{(5)}(d)=\{x_0x_2, 
x_0x_1+x_1^2, 
x_0^2, 
x_1^3+dx_1^2x_2,  
x_1^2x_2x_3,
x_1^2x_2^2\}$ where $d\in A$.
 
\noindent In conclusion, for every $k=1,\dots,5$, the liftings of $I'$ in $\St_{J^{(k)}}$ are the $x_4$-liftings of the ideals generated by the corresponding Gr\"obner bases ${(G^{(k)}(a,b,c,d))}$, modulo the defining ideal of the Gr\"obner stratum over $J^{(k)}$.
\end{example}

\begin{example}\label{ex:lCM} 
Consider a field $K$ with $\mathrm{char}(K)=0$ and a $K$-algebra $A$. In the present example we test our methods in order to characterize the locally Cohen-Macaulay liftings in $\mathbb P_A^3$ of the double point $Y\subset \mathbb P^2_K$ defined by the ideal $I'=(x_0,x_1^2)\subset K[x_0,x_1,x_2]$, where $J'=\In(I')=I'$ is quasi-stable. For every Hilbert polynomial $p(t)$, a lifting $W$ of $Y$ with Hilbert polynomial $p(t)$ has embedded or isolated points if and only if $W$ contains a lifting with Hilbert polynomial $p(t)-1$. 

For $p(t)=2t+1$ we find a unique monomial quasi-stable lifting $J=(x_0,x_1^2)\subset A[x_0,x_1,x_2,x_3]$ of $J'$.  Furthermore, $H=I'\subset K[x_0,x_1,x_2]$ is the unique ideal such that $H^\sat=I'$ and $\mathrm{in}(H)=J$.
Hence, $\underline{\mathrm{L}}_Y^{2t+1}(A)$ consists of the $x_3$-liftings of the ideal $(x_0,x_1^2)$. These liftings are generated by the following Gr\"obner basis, which depends on the free parameters $\alpha, \beta, \gamma, \delta \in A$:
\[
G(\alpha,\beta,\gamma,\delta)=\{x_{{0}}+\alpha\,x_{{3}},x_{{1}}^{2}+\beta\,x_{{1}}x_{{3}}+
\gamma\,x_{{2}}x_{{3}}+\delta\,x_{{3}}^{2}\}.
\]

For $p(t)=2t+2$, we find two monomial quasi-stable liftings $J^{(1)}=(x_0,x_1^3,x_1^2x_2)$ and $J^{(2)}=(x_0^2,x_0x_1,x_1^2,x_0x_2)$ of $J'$. By direct computation we see that $J^{(1)}$ (resp.~$J^{(2)}$) is the unique ideal in $K[x_0,x_1,x_2]$ such that its saturation is $I'$ and its initial ideal is $J^{(1)}$ (resp.~$J^{(2)}$). 
Hence, $\underline{\mathrm{L}}_Y^{2t+2}(A)$ consists of the $x_3$-liftings of the ideals $J^{(1)}$ and $J^{(2)}$. We now study them in order to find the locally Cohen-Macaulay liftings of $Y$ with Hilbert polynomial $2t+2$.

Every $x_3$-lifting of $J^{(1)}$ 
defines a plane curve with isolated or embedded points and hence is not locally Cohen-Macaulay. 

Every $x_3$-lifting of $J^{(2)}$ is generated by polynomials of the following type
\begin{equation}\label{eq:ex75groebner}
\begin{array}{c}
x_{{0}}^{2}+c_{{1}}x_{{0}}x_{{3}}+ \left( c_{{1}}c_{{8}}-c_{{8}}^{2} \right) x
_{{3}}^{2},\quad
x_{{0}}x_{{1}}+c_{{2}}c_{{8}}x_{{3}}^{2}+c_{{2}}x_{{0}}x_
{{3}}+c_{{8}}x_{{1}}x_{{3}},\\
x_{{1}}^{2}+c_{{3}}x_{{0}}x_{{3}}+c_{{4
}}x_{{1}}x_{{3}}+c_{{5}}x_{{2}}x_{{3}}+ \left( c_{{1}}c_{{3}}-c_{{2}
}^{2}+c_{{2}}c_{{4}}-c_{{3}}c_{{8}}+c_{{5}}c_{{6}} \right) x_{{3}}^{
2},\\
x_{{0}}x_{{2}}+c_{{6}}x_{{0}}x_{{3}}+c_{{7}}x_{{1}}x_{{3}}+c_{{8}}x_
{{2}}x_{{3}}+ \left(\frac{c_{{4}}c_{{7}}}{2}+c_{{6}}c_{{8}} \right) {x_{{3
}}}^{2}
\end{array}
\end{equation}
with $c_1,\dots,c_8\in A$ such that $c_{{5}}c_{{7}}=c_{{3}}c_{{7}}=2\,c_{{2}}c_{{7}}-c_{{4}}c_{{7}}=c_
{{1}}c_{{7}}-2\,c_{{7}}c_{{8}}=0$.
The assumptions on the $c_i$'s ensure that the polynomials in \eqref{eq:ex75groebner} form a Gr\"obner basis.

A $x_3$-lifting of $J^{(2)}$ is contained in the ideal generated by $G(\alpha, \beta, \gamma, \delta)$, with $\alpha, \beta, \gamma, \delta \in A$, if and only if
\begin{equation}\label{eq:ex75nonlCM}
 c_{{1}}=c_8=\alpha,c_{{4}}=\beta,c_{{5}}=\gamma,c_{{2}}=c_{{3}}=c_{{6}}=c_{{7}}=0. 
 \end{equation}
Observe that $\delta=0$. Then,  the conditions \eqref{eq:ex75nonlCM} characterize the liftings of the double point $Y$ which are not locally Cohen-Macaulay.

Summing up, the locally Cohen-Macaulay liftings of $I'$ with Hilbert polynomial $2t+2$ are $x_3$-liftings of $J^{(2)}$ having Gr\"obner basis as in \eqref{eq:ex75groebner} and such that the $c_i$'s do not satisfy equations~\eqref{eq:ex75nonlCM}.
\end{example}


\section{Backgroud II: marked functor over a truncation ideal}\label{sec:backgroundMarkedBases}

The results of Section \ref{sec:quasi-stable} are preliminary for the proof of the representability of $\underline{\mathrm{L}}_Y^{p(t)}$, which is described in Section \ref{sec:reprLYp} and uses the notion of marked functor. First, we need to recall the notion of Pommaret basis.

\begin{definition}\citep{Seiler2009I}\label{def:Pommaret}
Given a term $x^\alpha$ and $x_j=min(x^\alpha)$, the set $\mathcal C_P(x^\alpha):=\{x^\delta x^\alpha \ \vert \ \delta_i=0 \ \forall i<j\}$ is the \emph{Pommaret cone} of $x^\alpha$. Given a finite set $M$ of terms, its \emph{Pommaret span} is $\cup_{x^\alpha\in M} \mathcal C_{\mathcal P}(x^\alpha)$. The finite set of terms $M$ is a \emph{Pommaret basis $\mathcal P((M))$} of the ideal $(M)$ if the Pommaret span of $M$ coincides with the set of terms in $(M)$ and the Pommaret cones of the terms in $M$ are pairwise disjoint. 
\end{definition}

A monomial ideal $J$ is quasi-stable if and only if it has a Pommaret basis $\mathcal P(J)$ (see \cite[Definition 4.3 and Proposition 4.4]{Seiler2009II}). A {\em marked polynomial} is a polynomial $F$ together with a specified term of $\supp(F)$ that will be called {\em head term of $F$} and denoted by $\Ht(F)$ \citep{RStu}. 

\begin{definition}\label{def:cose marcate} \cite[Definition 5.1]{CMR13}
Let $J\subset \Aox$ be a quasi-stable ideal.

A {\em $\mathcal P(J)$-marked set} (or marked set over $\mathcal P(J)$) $\mathcal G$ is a set of homogeneous monic marked polynomials $F_\alpha$ in $\Aox$ such that  the head terms $Ht(F_\alpha)=x^\alpha$ are pairwise different and form the Pommaret basis $\mathcal P(J)$ of $J$, and $\supp(F_\alpha-x^\alpha)\subset \mathcal N(J)$.

A $\mathcal P(J)$\emph{-marked basis} (or marked basis over $\mathcal P(J)$) $\mathcal G$ is a $\mathcal P(J)$-marked set such that $\mathcal N(J)$ is a basis of $\Aox/(\mathcal G)$ as an $A$-module, i.e.~$\Aox=(\mathcal G)\oplus \langle \cN (J) \rangle$ as an $A$-module.
%
\end{definition} 

Let $J\subset A[\mathbf x,x_n]$ be a saturated quasi-stable ideal and $m$ a non-negative integer. The ideal $J_{\geq m}$ is called {\em $m$-truncation} of $J$ and is quasi-stable, because $J$ is.
Referring to \citep{CMR13,LR2,BCR2}, we now recall the definition of marked functor over the $m$-truncation of a saturated quasi-stable ideal. 

\begin{definition}\label{def:markedscheme}\citep{BCR2,LR2} 
The {\em marked functor} \ $\underline{\mathrm{Mf}}_{\mathcal P(J_{\geq m})}:  \underline{\text{Noeth-}K\text{-Alg}}\rightarrow \underline{\mathrm{Sets}}$ \ 
associates to every Noetherian $K$-algebra $A$ the set
$$\underline{\mathrm{Mf}}_{\mathcal P(J_{\geq m})} (A) :=\{I\text{ saturated ideal in }A[\mathbf x,x_n]\  : \ A[\mathbf x,x_n]=I_{\geq m}\oplus\langle \mathcal N(J_{\geq m})\rangle\} =$$
$$= \{I\text{ saturated ideal in }A[\mathbf x,x_n]\  : I_{\geq m} \text{ is generated by a } \mathcal P(J_{\geq m})\text{-marked basis}\}$$
and to every $K$-algebra morphism  $\phi : A\rightarrow B$ the function  
$${\phi} \colon \underline{\mathrm{Mf}}_{\mathcal P(J_{\geq m})}(A) \rightarrow \underline{\mathrm{Mf}}_{\mathcal P(J_{\geq m})}(B) \ \hbox{ given by } \  {\phi} (I )= I\otimes_A B.$$
\end{definition}

The marked functor $\underline{\mathrm{Mf}}_{\mathcal P(J_{\geq m})}$ is a {\em representable subfunctor} of the Hilbert functor $\underline{\mathrm{Hilb}}^n_{p(t)}$, where $p(t)$ is the Hilbert polynomial of $A[\mathbf x,x_n]/J$, and we denote by $\mathrm{Mf}_{\mathcal P(J_{\geq m})}$ its representing scheme \cite[Theorem 6.6]{BCR2}.

\begin{remark} 
The fact that $I$ belongs to $\underline{\mathrm{Mf}}_{\mathcal P(J_{\geq m})} (A)$ {\em does not imply} that $I$ belongs to $\underline{\mathrm{Mf}}_{\mathcal P(J)} (A)$ (e.g. \cite[Example 3.8]{BCR2}). If $I$ belongs to $\underline{\mathrm{Mf}}_{\mathcal P(J_{\geq m})} (A)$, then the Hilbert function of $\Aox/I_{\geq m}$ is equal to the Hilbert function of $\Aox/J_{\geq m}$.
\end{remark}

\begin{remark} 
The original definition of marked functor over a $m$-truncation of a quasi-stable ideal introduced in \citep{LR2} stated: 
$$\underline{\mathrm{Mf}}_{\mathcal P(J_{\geq m})} (A) :=\{M\text{ homogeneous ideal in }A[\mathbf x,x_n]\  : \ A[\mathbf x,x_n]=M\oplus\langle \mathcal N(J_{\geq m})\rangle\}.$$
Our presentation is equivalent to the original one because a homogeneous ideal $M\subset A[\mathbf x,x_n]$ such that $A[\mathbf x,x_n]=M\oplus\langle \mathcal N(J_{\geq m})\rangle$ is exactly of type $I_{\geq m}$, where $I=M^{\sat}$ (see \cite[Corollary 3.7]{BCR2}).
\end{remark}

\section{The functor $\underline{\mathrm{L}}_Y^{p(t)}$ is representable}\label{sec:reprLYp}

We can now constructively prove that the functor $\underline{\mathrm{L}}_Y^{p(t)}$ is representable. 

From now, $I'\subset K[\mathbf x]$ is a homogeneous saturated ideal with $J'=\mathrm{in}(I')$ quasi-stable, $J\subset A[\mathbf x,x_n]$ is a monomial lifting of $J'$ and $Y$ is the projective subscheme defined by $I'$.

\begin{definition}\label{def:sottofuntori aperti}
We denote by $\F$ the subfunctor $\underline{\mathrm{L}}_Y^{p(t)}\cap \underline{\mathrm{Mf}}_{\mathcal P(J_{\geq m})}:  \underline{\text{Noeth-}K\text{-Alg}}\rightarrow \underline{\mathrm{Sets}}$ \ of $\underline{\mathrm{Hilb}}^n_{p(t)}$ associating to every Noetherian $K$-algebra $A$ the set
$\F(A) := \underline{\mathrm{L}}_Y^{p(t)}(A)~\cap~\underline{\mathrm{Mf}}_{\mathcal P(J_{\geq m})}~(A)$ 
and to every $K$-algebra morphism  $\phi : A\rightarrow B$ the function  
${\phi} \colon \F(A) \rightarrow \F(B)$ given by ${\phi}(I )= I\otimes_A B$.
\end{definition}

\begin{lemma} \label{lemma:sottofuntore aperto}
Let $I'\subset K[\mathbf x]$ be a homogeneous saturated ideal with $J'=\mathrm{in}(I')$ quasi-stable, $J\subset A[\mathbf x,x_n]$ a monomial lifting of $J'$ and $\rho$ the satiety of $(J,x_n)/(x_n)\subset A[\mathbf x]$. If $m\geq \rho$ then:
\begin{enumerate}[(i)]
\item A saturated ideal $I\subset A[\mathbf x,x_n]$ belongs to $\underline{\mathrm{L}}_{Y,J_{\geq m-1}}^{p(t)}(A)$ if and only if $I$ is a lifting of $I'$ and $I_{\geq m-1}$ has a $\mathcal P(J_{\geq m-1})$-marked basis.
\item\label{lemma:sottofuntore aperto_ii} $\underline{\mathrm{L}}_{Y,J_{\geq m-1}}^{p(t)}$ is an open subfunctor of $\underline{\mathrm{L}}_Y^{p(t)}$. 
\item \label{lemma:sottofuntore aperto_iii} $\underline{\mathrm{L}}_{Y,J_{\geq m}}^{p(t)}$ is isomorphic to $\underline{\mathrm{L}}_{Y,J_{\geq m-1}}^{p(t)}$.
\end{enumerate}
\end{lemma}

\begin{proof}
The statements are direct consequences of Definition \ref{def:sottofuntori aperti} and of properties of marked functors. In particular, item \eqref{lemma:sottofuntore aperto_ii} is a consequence of the openness of  $\underline{\mathrm{Mf}}_{\mathcal P(J_{\geq m})}$ in $\underline{\mathrm{Hilb}}^n_{p(t)}$ (see \citep{LR2} and \cite[Proposition 6.13]{BCR2}); item \eqref{lemma:sottofuntore aperto_iii} follows from the analogous property of marked functors (see \cite[Theorem 5.7]{BCLR}, \cite[Theorem 3.4]{LR2} if $J$ is strongly stable and \cite[Section 6]{BCR2} in the more general case $J$ is quasi-stable).
\end{proof}

\begin{lemma}\label{lemma:Pommaret}
Let $J'\subset K[\mathbf x]$ be a saturated quasi-stable ideal, $J\subset A[\mathbf x,x_n]$ a monomial lifting of $J'$ and $\rho$ the satiety of $(J,x_n)/(x_n)\subset A[\mathbf x]$. If $m\geq \rho$ then: 
\begin{itemize}
\item[(i)]$\left(\frac{(J,x_n)}{(x_n)}\right)_{\geq m}={J'}_{\geq m}\otimes A$, and 
\item[(ii)] $\mathcal P({J'}_{\geq m})= \{ x^\alpha \in \mathcal P(J_{\geq m}) \ : \  x^\alpha \hbox{ is not divisible by } x_n\}$.
\end{itemize}
\end{lemma}

\begin{proof}
Observe that the terms of the Pommaret basis of $J$ are not divisible by $x_n$ \cite[Remark 2.8]{BCR2}. Moreover, recall that in our setting we have $J'\otimes A:=\left(\frac{(J,x_n)}{(x_n)}\right)^\sat$, where the ideal $\frac{(J,x_n)}{(x_n)}=(\mathcal P(J))\cdot A[\mathbf x]$ is quasi-stable and $m$-saturated because $m\geq \rho$. Thus, we obtain ${J'}_{\geq m}\otimes A=\left(\frac{(J,x_n)}{(x_n)}\right)_{\geq m}$ and then use the definition of Pommaret basis. 
\end{proof}

\begin{lemma}\label{lemma:sezioni ideali marcati}
Let $J'\subset K[\mathbf x]$ be a saturated quasi-stable ideal, $J\subset A[\mathbf x,x_n]$ a monomial lifting of $J'$ and $\rho$ the satiety of $(J,x_n)/(x_n)\subset A[\mathbf x]$. For every $m\geq \rho$, if $I\subset A[\mathbf x,x_n]$ belongs to $\underline{\mathrm{Mf}}_{\mathcal P(J_{\geq m-1})}(A)$, then $\left(\frac{(I,x_n)}{(x_n)}\right)_{\geq m}$ belongs to $\underline{\mathrm{Mf}}_{\mathcal P({J'}_{\geq m})}(A)$.
\end{lemma}

\begin{proof}
For every $t\geq m$ we have 
\begin{equation}\label{eq:uguaglianza Hilbert function}
h_{A[\mathbf x,x_n]/(I,x_n)}(t)=h_{A[\mathbf x,x_n]/(J,x_n)}(t)=h_{K[\mathbf x]/J'}(t).
\end{equation}
Indeed, $h_{A[\mathbf x,x_n]/I}(t-1)= h_{A[\mathbf x,x_n]/J}(t-1)$ because $I_{\geq m-1}$ has a $\mathcal P(J_{\geq m-1})$-marked basis by hypothesis; for every $t\geq 0$, we have $h_{A[\mathbf x,x_n]/(I,x_n)}(t)=\Delta h_{A[\mathbf x,x_n]/I}(t)$ and $h_{A[\mathbf x,x_n]/(J,x_n)}(t)=\Delta h_{A[\mathbf x,x_n]/J}(t)$, because $J$ and $I$ are saturated in $A[\mathbf x,x_n]$ and $x_n$ is generic for $J$ and also for $I$, by \cite[Theorem 3.5 and Corollary 3.7]{BCR2}; for every $t\geq m$, we obtain $h_{A[\mathbf x,x_n]/(J,x_n)}(t)=h_{K[\mathbf x]/J'}(t)$, because $J$ is a lifting of $J'$ and then $\left(\frac{(J,x_n)}{(x_n)}\right)_{\geq m}={J'}_{\geq m}\otimes A$ by Lemma \ref{lemma:Pommaret}.

A $\mathcal P(J_{\geq m-1})$-marked set $\mathcal G$ of $I_{\geq m-1}$ is of the following type:
\begin{equation}\label{eq:base marcata Im-1}
\mathcal G =\left \{f_\alpha=x^\alpha +  \sum_{\mathcal N(J)_{\vert \alpha\vert} \ni x^\gamma} C_{\alpha\gamma} x^\gamma \ : \
\Ht(f_\alpha)=x^\alpha\in \mathcal P(J_{\geq m-1})\right\}\subset K[\mathcal C_{J_{\geq m-1}}][\mathbf x,x_n],
\end{equation}
where $\mathcal C_{J_{\geq m-1}}$ is the set of the variables $C_{\alpha\gamma}$. 
Let $\mathcal A_{J_{\geq m-1}}\subseteq K[\mathcal C_{J_{\geq m-1}}]$ be the defining ideal of $\mathrm{Mf}_{\mathcal P(J_{\geq m-1})}$, that is the ideal exactly generated by all the relations which are satisfied by the variables $C_{\alpha\gamma}$ in order to obtain that $\mathcal G$ is a $\mathcal P(J_{\geq m-1})$-marked basis. Indeed, the functor $\underline{\mathrm{Mf}}_{\mathcal P(J_{\geq m-1})}$ is represented by the scheme $\Spec\Bigl({K[\mathcal C_{J_{\geq m-1}}]}/{\mathcal A_{J_{\geq m-1}}}\Bigr)$.
We now construct a $\mathcal P({J'}_{\geq m})$-marked basis for the ideal $\left(\frac{(I,x_n)}{(x_n)}\right)_{\geq m}$ by taking Lemma \ref{lemma:Pommaret}  into account.  

For every $x^\alpha\in \mathcal P(J_{\geq m})_m$ that is not divisible by $x_n$, we consider its $J_{\geq m-1}$-reduced form $\Nf(x^\alpha)$ modulo $I_{\geq m-1}$  and the homogeneous polynomial $g_\alpha:=x^\alpha-\Nf(x^\alpha)\in I_{\geq m-1}$. Observe that $\Nf(x^\alpha)$ exists for properties of marked bases. Then, we take the following set that is a $\mathcal P({J'}_{\geq m})$-marked set by Lemma \ref{lemma:Pommaret} and \cite[Proposition 2.10]{BCR2}
\begin{equation}\label{eq:base marcata Im}
\mathcal G':=\{g_\alpha|_{x_n=0} \ : \ x^\alpha\in \mathcal P(J_{\geq m})_m \text{ and } x^\alpha \text{ not divisible by } x_n \}\cup\{f_\beta|_{x_n=0} \ : \ f_\beta \in \mathcal G_{\geq m+1}\}.
\end{equation}
The polynomials of $\mathcal G'$ belong to $(I,x_n)/(x_n)$. 
By construction, ${\mathcal G'}_m$ is a linearly independent set of polynomials and so $h_{A[\mathbf x,x_n]/(J,x_n)}(m)=h_{A[\mathbf x]/J'}(m)=h_{A[\mathbf x]/(\mathcal G')}(m)$. For every $t\geq m+1$, by induction we suppose $\left((I,x_n)/(x_n)\right)_{t-1}$ is equal to $({\mathcal G'})_{t-1}$. We have $I_{t}=A[\mathbf x, x_n]_1\cdot I_{t-1}+\langle \mathcal G_{t}\rangle$, hence 
\[
\begin{array}{lll}
\left((I,x_n)/(x_n)\right)_{t}&=A[\mathbf x]_1\cdot \left( (I,x_n)/(x_n)\right)_{t-1}+\langle {\mathcal G'}_{t}\rangle & \\
& =A[\mathbf x]_1\cdot (\mathcal G')_{t-1}+\langle {\mathcal G'}_{t}\rangle & =(\mathcal G')_t.
\end{array}
\]
Then, $\mathcal G'$ is a $\mathcal P({J'}_{\geq m})$-marked set and generates $\left(\frac{(I,x_n)}{(x_n)}\right)_{\geq m}$. Moreover, by \cite[Corollary 4.9]{BCR2} $\mathcal G'$ is a $\mathcal P({J'}_{\geq m})$-marked basis because $h_{A[\mathbf x,x_n]/(I,x_n)}(t)=h_{K[\mathbf x]/J'}(t)$ for every $t\geq m$, like observed in \eqref{eq:uguaglianza Hilbert function}. Thus, $\left(\frac{(I,x_n)}{(x_n)}\right)_{\geq m}$ belongs to $\underline{\mathrm{Mf}}_{\mathcal P({J'}_{\geq m})}(A)$.
\end{proof}

\begin{remark}\label{rem:insufficienza marked schemes}
Marked schemes do not give us a characterization of liftings because ${I'}_{\geq m}\in \underline{\mathrm{Mf}}_{\mathcal P({J'}_{\geq m})}(K)$ does not imply $I'\in \underline{\mathrm{Mf}}_{\mathcal P(J')}(K)$ \cite[Example 3.8]{BCR2} \label{ex:BCR2}.
\end{remark}

In the proof of Lemma \ref{lemma:sezioni ideali marcati} we have denoted by $\mathcal A_{J_{\geq m-1}}\subseteq K[\mathcal C_{J_{\geq m-1}}]$ the defining ideal of $\mathrm{Mf}_{\mathcal P(J_{\geq m-1})}$. In the same hypothesis of Lemma \ref{lemma:sezioni ideali marcati}, we can also consider the $\mathcal P({J'}_{\geq m})$-marked set $\mathcal G'$ of $\left(\frac{(I,x_n)}{(x_n)}\right)_{\geq m}$ which is given in formula \eqref{eq:base marcata Im} and which is a marked basis under the conditions posed by the ideal $\mathcal A_{J_{\geq m-1}}$. 

Let $\mathcal F$ be the set of the ${J'}_{\geq m}$-normal forms modulo $I'$ of the polynomials in $\mathcal G'$ (which can be computed by the reduced Gr\"obner basis of $I'$ with respect to degrevlex) and denote by $\mathcal B_{{J'}_{\geq m}}\subseteq K[\mathcal C_{J_{\geq m-1}}]$ the ideal generated by the coefficients of the terms in the polynomials of $\mathcal F$. Thus, $(\mathcal G')=((I,x_n)/(x_n))_{\geq m}$ is contained in $I'$ under the conditions posed by the ideal $\mathcal B_{{J'}_{\geq m}}$. 

In conclusion, the ideal $\mathcal A_{J_{\geq m-1}}+\mathcal B_{{J'}_{\geq m}}$ exactly gives the conditions that $\mathcal G'$ is a marked basis and that $(\mathcal G')=((I,x_n)/(x_n))_{\geq m}$ is contained in $I'$. 

\begin{theorem}\label{th:aperti}
Let $J'\subset K[\mathbf x]$ be a saturated quasi-stable ideal, $J\subset A[\mathbf x,x_n]$ a monomial lifting of $J'$ and $\rho$ the satiety of $(J,x_n)/(x_n)\subset A[\mathbf x]$.  For every $m\geq \rho$, 
the functor $\underline{\mathrm{L}}_{Y,J_{\geq m-1}}^{p(t)}$ is representable by $\Spec\Bigl({K[\mathcal C_{J_{\geq m-1}}]}/{(\mathcal A_{J_{\geq m-1}}+\mathcal B_{{J'}_{\geq m}})}\Bigr)$.
\end{theorem}

\begin{proof}
By Lemma \ref{lemma:sottofuntore aperto}(i), a saturated ideal $I$ belongs to $\underline{\mathrm{L}}_{Y,J_{\geq m-1}}^{p(t)}(A)$ if and only if it is a lifting of $I'$ and $I_{\geq m-1}$ has a $\mathcal P(J_{\geq m-1})$-marked basis. 
By Lemma \ref{lemma:sezioni ideali marcati}, if $I\subset A[\mathbf x,x_n]$ is a saturated ideal such that $I_{\geq m-1}$ has a $\mathcal P(J_{\geq m-1})$-marked basis $\mathcal G$, we can assume that $\mathcal G$ is of type  \eqref{eq:base marcata Im-1} under the conditions posed by the ideal $\mathcal A_{J_{\geq m-1}}$. 
Then, the ideal $\mathcal A_{J_{\geq m-1}}+\mathcal B_{{J'}_{\geq m}}$ gives the conditions that the ideal $\left(\frac{(I,x_n)}{(x_n)}\right)_{\geq m}$ has a $\mathcal P({J'}_{\geq m})$-marked basis $\mathcal G'$, like in \eqref{eq:base marcata Im}, so that $(\mathcal G')=((I,x_n)/(x_n))_{\geq m}$ is contained in $I'$. Hence, $\left(\frac{(I,x_n)}{(x_n)}\right)_{\geq m}$ is $m$-saturated, i.e.~$\left(\frac{(I,x_n)}{(x_n)}\right)_{\geq m}=\left(\left(\frac{(I,x_n)}{(x_n)}\right)^\sat\right)_{\geq m}$, by \cite[Corollary 3.7]{BCR2}. Now, we have $\Bigl((I,x_n)/(x_n)\Bigr)^{\sat}=I'$ by Hilbert polynomial arguments, hence $I$ is a lifting of $I'$ and we can conclude.
\end{proof}

In Algorithm \ref{alg:computation in MS} we collect the main steps of the method applied in the proofs of Lemma \ref{lemma:sezioni ideali marcati} and of Theorem \ref{th:aperti} in order to construct the parameter schemes $\mathrm{L}_{Y,J_{\geq m-1}}^{p(t)}$.

\begin{algorithm}[!ht]
\caption{\label{alg:computation in MS} Algorithm for computing the parameter schemes $\mathrm{L}_{Y,J_{\geq m-1}}^{p(t)}$ for the liftings of a saturated homogeneous ideal $I'\subset K[\mathbf x]$ over a Noetherian $K$-algebra $A$ in $\mathrm{Hilb}_{p(t)}^n$.}
\begin{algorithmic}[1]
\STATE $\textsc{LiftingMS}\big(I',p(t)\big)$
\REQUIRE $I'\subset K[\mathbf x]$ a saturated polynomial ideal with $\In(I')$  quasi-stable.
\REQUIRE  $p(t)$ a Hilbert polynomial such that $\Delta p(t)= p_Y(t)$, where $p_Y(t)$ is the Hilbert polynomial of the scheme $Y$ defined by $I'$.
\ENSURE  A set $\mathfrak B$ containing parameter schemes for the liftings of $I'$ in $\mathrm{L}_{Y,J_{\geq m-1}}^{p(t)}$, for every quasi-stable lifting $J\subset \Aox$ of $\In(I')$ with Hilbert polynomial $p(t)$.
\STATE $\mathcal L:=\{ J \subset \Aox \ \vert \ J \text{ quasi-stable lifting of } \In(I') \text{ with Hilbert polynomial } p(t)\}$; 
\STATE $\mathfrak B=\emptyset$;
\FOR{$J \in \mathcal L$}
\STATE $m:=\varrho=\sat((J,x_n)/(x_n))$;
\STATE let $\mathcal A_{J_{\geq m-1}}\subseteq K[\mathcal C_{J_{\geq m}}]$ be the defining ideal of $\mathrm{Mf}_{\mathcal P(J_{\geq m-1})}$;
\STATE let $\mathcal G'$ be the set of formula \eqref{eq:base marcata Im};
\STATE $\mathcal B_{{J'}_{\geq m}}:=(0)$;
\FOR{$g\in \mathcal G'$}
\STATE $\mathcal B_{{J'}_{\geq m}}:=\mathcal B_{{J'}_{\geq m}}+(\mathrm{coeff}(\overline{g}))$, where $\overline{g}$ is the ${J'}_{\geq m}$-normal form of $g$ modulo $I'$; 
\ENDFOR
\STATE $\mathfrak B:=\mathfrak B\cup \{\mathcal A_{J_{\geq m-1}}+\mathcal B_{{J'}_{\geq m}}\}$;
\ENDFOR
\end{algorithmic}
\end{algorithm}

\begin{remark}
Unlike 
the ideal $\mathfrak b_{J\cap K[\mathbf x]}$ in the proof of Theorem \ref{th:parametrizzazione}, the ideal $\mathcal B_{{J'}_{\geq m}}\subseteq K[\mathcal C_{J_{\geq m-1}}]$ 
is generated by polynomials which are not necessarily linear. In order to have linear generators, we could repeat the proof of Theorem \ref{th:parametrizzazione} starting from the $\mathcal P(J_{\geq m})$-marked basis of $I_{\geq m}$ instead of  the $\mathcal P(J_{\geq m-1})$-marked basis of $I_{\geq m-1}$.
However, the number of variables $\mathcal C_{J_{\geq m}}$ is much higher than the number of the variables $\mathcal C_{J_{\geq m-1}}$: indeed, the ideal $\mathcal A_{J_{\geq m-1}}$ can be obtained from $\mathcal A_{J_{\geq m}}$ by an elimination of variables (which is a time-consuming process). This elimination of variables applied on the linear generators of $\mathcal B_{J'_{\geq m}}\subseteq K[\mathcal C_{J_{\geq m-1}}]$ gives generators of higher degree in a smaller number of variables.  If we use the process of the proof of Theorem \ref{th:aperti}, the eliminable variables do not appear from the very beginning, allowing the embedding of $\underline{\mathrm{L}}_{Y,J_{\geq m-1}}^{p(t)}$ in a rather small affine space, without an expensive elimination of variables (see \cite[Section 5.1]{BCLR} for details).
\end{remark}

\begin{theorem}\label{th:funtore di punti}
$\underline{\mathrm{L}}_Y^{p(t)}$ is representable.
\end{theorem}

\begin{proof}
We check that $\underline{\mathrm{L}}_Y^{p(t)}$ is a Zariski sheaf and that there are rings $R_i$ and open subfunctors $\alpha_i:h_{R_i}\rightarrow \underline{\mathrm{L}}_Y^{p(t)}$ such that, for every field $\mathbb K\supseteq K$, $\underline{\mathrm{L}}_Y^{p(t)}(\mathbb K)$ is the union of the images of $h_{R_i}(\mathbb K)$ under $\alpha_i$. Indeed, these two conditions imply that $\underline{\mathrm{L}}_Y^{p(t)}$ is isomorphic to the functor of points of a scheme, by \cite[Theorem VI-14]{EH}.

We already observed that $\underline{\mathrm{L}}_Y^{p(t)}$ is a Zariski sheaf (Proposition \ref{prop:Zariski sheaf}). 
For the other condition, due to Theorems \ref{th:aperti} and \ref{th:lifting}, we can consider the finite set of the open subfunctors $\underline{\mathrm{L}}_{Y,J_{\geq m-1}}^{p(t)}$ of the functor $\underline{\mathrm{L}}_Y^{p(t)}$, where $J$ varies among all the quasi-stable ideals in $K[\mathbf x,x_n]$ that are liftings of $J'$ with $p(t)$ as Hilbert polynomial of $A[\mathbf x,x_n]/J$.

By Theorem \ref{th:aperti}, for every of these ideals $J$, the open functor $\underline{\mathrm{L}}_{Y,J_{\geq m-1}}^{p(t)}$ is the functor of points of the ring $R_J:=K[\mathcal C_{J_{\geq m}}]/(\mathcal A_{J_{\geq m-1}}+\mathcal B_{{J'}_{\geq m}})$.

As recalled in Theorem \ref{th:main features}(ii), for every non-negative integer $m$ and every $K$-algebra $A$ we have $\underline{\St}_{J}(A) \ \simeq \ \underline{\St}_{{J}_{\geq m}}(A)$, with the degrevelex order, and hence, 
\begin{equation}\label{eq:contenuti funtori}
\underline{\St}_{J}(A) \ \simeq \ \underline{\St}_{{J}_{\geq m}}(A) \ \subseteq \ \underline{\mathrm{Mf}}_{\mathcal P({J}_{\geq m})}(A).
\end{equation}
Thus, thanks to Theorem \ref{th:dove}, for every $K$-algebra $A$ we obtain $\underline{\mathrm{L}}_Y^{p(t)}(A)=\cup_J \ \underline{\mathrm{L}}_{Y,J_{\geq m-1}}^{p(t)}(A)$. This holds in particular for $A=\mathbb K\supseteq K$.
\end{proof}

\begin{remark}\label{rem:stromme} 
Concerning the proof of Theorem \ref{th:funtore di punti}, we can observe that the open subfunctors $\underline{\mathrm{L}}_{Y,J_{\geq m-1}}^{p(t)}$ form an open covering of the functor $\underline{\mathrm{L}}_Y^{p(t)}$ (e.g.~\cite[Definition 2.15]{Stromme}). Thus, \cite[Proposition 2.16]{Stromme}  clarifies that a scheme defining a functor of points isomorphic to $\underline{\mathrm{L}}_Y^{p(t)}$ can be constructed by means of fiber products of the subfunctors $\underline{\mathrm{L}}_{Y,J_{\geq m-1}}^{p(t)}$.
\end{remark}

In next example we show that there are liftings in a marked scheme over a quasi-stable ideal that do not belong to the Gr\"obner stratum over the same ideal.

\begin{example}\label{ex:esempio bello}
Consider the scheme $Y$ defined by the ideal $I'\subset K[x_0,\dots,x_3]=K[\mathbf x ]$ of Example \ref{ex:un primo calcolo} and the quasi-stable ideal $J^{(2)}=(x_2x_0, x_0^2, x_3x_1x_0, x_2x_1^2, x_1^3, x_1^2x_0)$. 
Following the proof of Lemma \ref{lemma:sezioni ideali marcati} and the proof of Theorem \ref{th:aperti}, we can construct the Noetherian $K$-algebra $K[\mathcal C_{J^{(2)} _{\geq m-1}}]/(\mathcal A_{J^{(2)}_{\geq m-1}}+\mathcal B_{{J'}_{\geq m}})$ defining the functor $\underline{\mathrm{L}}_{Y,J^{(2)}_{\geq m-1}}^{p(t)}$, with $p(t)=t^2+4t+1$, and $m=3$. 
Then, we observe that there are some liftings of $Y$ which belong to $\underline{\mathrm{L}}_{Y,{J^{(2)}}_{\geq m-1}}^{p(t)}(K)$ but do not belong to $\underline{\St}_{{J^{(2)}}_{\geq m-1}}(K)$. For instance, for every $e\in K, e\neq 0$, the ideal
\[
I=(x_2x_0-ex_1^{2}-ex_1x_0,x_0^{2},x_3
x_1^{2}+x_3x_1x_0,x_2x_1^{2},x_1^{3},
x_1^{2}x_0,x_0x_1x_2)
\]
is a lifting of $I'$, it has a marked basis over ${J^{(2)}}_{\geq 2}$, hence it belongs to $\underline{\mathrm{L}}_{Y,{J^{(2)}}_{\geq m-1}}^{p(t)}(K)$, but $\In(I)=J^{(3)}$: this means that $I$ does not belong to $\underline{\St}_{{J^{(2)}}_{\geq m-1}}(K)$. Moreover we observe that the marked basis of $I$ cannot be a Gr\"obner basis with respect to some term order, because  the term $x_1x_0$ is higher than $x_2x_0$ with respect to any term order.
For more information about the computations performed in this example, at the url http://wpage.unina.it/cioffifr/MaterialeCoCoALiftingGeometrico and the interested reader can find lists of generators for the ideals $\mathcal A_{J^{(2)}_{\geq m-1}}$ and $\mathcal B_{{J'}_{\geq m}}$ and a code written for CoCoA-4.7.5 \citep{CoCoA} in order to construct a marked family.
\end{example}

\section*{Acknowledgements}

The investigations that gave rise to this paper began with a study shared with Margherita Roggero, to whom we are grateful for useful suggestions and stimulating discussions. We also thank the anonymous referees for useful comments and suggestions.

The first and second authors are members of GNSAGA (INdAM, Italy). The second author is partially supported by Prin 2015 (2015EYPTSB\_011 - {\em Geometry of Algebraic varieties}, Unit\`{a} locale Tor Vergata, CUP  E82F16003030006).

\bibliographystyle{elsarticle-harv}

\end{document}